\documentclass[12pt]{amsart}
\usepackage[margin=3cm]{geometry}

\usepackage[english]{babel}
\usepackage[normalem]{ulem}
\usepackage{xcolor}
\definecolor{TheoremRefColor}{HTML}{B22222}
\usepackage{hyperref}

\usepackage{amsmath}
\allowdisplaybreaks[2]
\usepackage{amsthm}
\usepackage{amssymb}
\usepackage{bbm}
\usepackage{ucs}
\usepackage{enumitem}
\usepackage{tikz}
\usepackage{tikz-cd}
\usepackage[utf8]{inputenc}
\usepackage[T1]{fontenc}
\usepackage{mathtools}
\usepackage{comment}
\usepackage{dsfont}
\usepackage{faktor}

\usepackage{etoolbox}
\usepackage{needspace}
\makeatletter

\patchcmd{\@setaddresses}{\/:\space}{~}{}{}
\patchcmd{\@setaddresses}{\/:\space}{~}{}{}
\patchcmd{\@setaddresses}{\/:\space}{~}{}{}

\def\l@subsection{\@tocline{2}{0pt}{2.75pc}{3pc}{}}
\makeatother

\hypersetup{
  colorlinks=true,
  linkcolor=TheoremRefColor,
  citecolor=magenta,
  urlcolor=blue
}

\makeatletter
\def\paper@styledref#1.#2\@nil#3{%
  \def\paper@refprefix{#1}%
  \ifdefstring{\paper@refprefix}{thm}{%
    \begingroup
    \setlength{\fboxsep}{1pt}%
    \setlength{\fboxrule}{0.35pt}%
    \fcolorbox{TheoremRefColor}{white}{%
      \hyperref[#3]{\paper@unboxedref*{#3}}}%
    \endgroup
  }{\paper@unboxedref{#3}}%
}
\AtBeginDocument{%
  \NewCommandCopy{\paper@unboxedref}{\ref}
  \RenewDocumentCommand{\ref}{s m}{%
    \IfBooleanTF{#1}{\paper@unboxedref*{#2}}{%
      \begingroup
      \edef\paper@refanchor{\getrefbykeydefault{#2}{anchor}{}}%
      \expandafter\paper@styledref\paper@refanchor.\@nil{#2}%
      \endgroup
    }%
  }%
}
\makeatother
\AtBeginDocument{
   \def\MR#1{}
}
\usepackage[all]{xy}

\newcommand{\C}{{\mathbb C}}

\newcommand{\Z}{{\mathbb Z}}
\newcommand{\N}{{\mathbb N}}
\newcommand{\F}{{\mathbb F}}
\newcommand{\T}{{\mathbb T}}

\newcommand{\cB}{{\mathcal B}}

\newcommand{\cD}{{\mathcal D}}
\newcommand{\cE}{{\mathcal E}}

\newcommand{\cH}{{\mathcal H}}

\newcommand{\cP}{{\mathcal P}}
\newcommand{\cN}{{\mathcal N}}

\newcommand{\cR}{{\mathcal R}}

\newcommand{\cU}{{\mathcal U}}

\newcommand{\id}{\mathrm{id}}
\newcommand{\du}{d_{\mathrm u}}
\usepackage{graphicx}

\newcommand{\Aut}{{\rm Aut}}

\newcommand{\MALG}{\operatorname{MALG}}

\newcommand{\Linf}{L^\infty(X,\nu)}

\newcommand{\norm}[1]{\lVert#1\rVert}

\numberwithin{equation}{section}

\newtheorem{thm}{Theorem}[section]
\newtheorem{cor}[thm]{Corollary}
\newtheorem{lem}[thm]{Lemma}

\newtheorem{definition}[thm]{Definition}
\newtheorem{example}[thm]{Example}
\newtheorem{remark}[thm]{Remark}

\newtheorem{proposition}[thm]{Proposition}

\theoremstyle{definition}

\title[Invariant Subalgebras]{Invariant von Neumann subalgebras of crossed products and tensor products}
\date{\today}
 \author[Amrutam]{Tattwamasi Amrutam}
\address{Institute of Mathematics of the Polish Academy of Sciences, ul. Sniadeckich 8, 00-656, Warszawa, Poland}
\email{tattwamasiamrutam@gmail.com}
\author[Dudko]{Artem Dudko}
\address{Institute of Mathematics of the Polish Academy of Sciences, ul.~\'Sniadeckich 8, 00--656 Warszawa, Poland}
\email{adudko@impan.pl}
\author[Jiang]{Yongle Jiang}
\address{
School of Mathematical Sciences, Dalian University of Technology, Dalian, 116024, China}
\email{yonglejiang@dlut.edu.cn}
\author[Skalski]{Adam Skalski}
\address{Institute of Mathematics of the Polish Academy of Sciences, ul.~\'Sniadeckich 8, 00--656 Warszawa, Poland}
\email{a.skalski@impan.pl}

\begin{document}
\raggedbottom
\begin{abstract}
Let $\Gamma$ be a countable discrete group acting on a tracial von
Neumann algebra $(Q,\tau_Q)$ in a trace-preserving way. We study
invariant von Neumann subalgebras of $Q\rtimes\Gamma$ and
of $L(\Gamma)\bar\otimes Q$, where the latter is equipped with the diagonal action.
We show that, under the pointwise weak mixing assumption, an
invariant subalgebra whose image under the canonical conditional expectation onto $Q$ is scalar
is necessarily contained in $L(\Gamma)$. Further, under relative weak mixing and central freeness
assumptions, we determine the structure of invariant subalgebras of the crossed product
with non-trivial intersection with $Q$. 
For p.m.p.\ actions of groups with trivial amenable radical and the invariant subalgebras
rigidity property, pointwise relative weak mixing gives a complete classification of invariant subalgebras
in the tensor product case.
We construct essentially free, prime, pointwise weakly mixing p.m.p.\ actions
of $\F_n$ for every $2\leq n\leq\infty$, as well as non-prime examples
with explicitly described factors.
\end{abstract}
\maketitle

{\footnotesize\tableofcontents}

\section{Introduction}

Let $\Gamma$ be a countable discrete group and let $\sigma$ be a trace-preserving action of $\Gamma$ on a tracial von Neumann algebra
$(Q,\tau_Q)$ with separable predual. We write $u_g$, $g\in\Gamma$, for
the canonical unitaries in the von Neumann algebraic crossed product
$Q\rtimes\Gamma$. A von Neumann subalgebra $\cP\subseteq Q\rtimes\Gamma$
is called \emph{$\Gamma$-invariant} if $u_g\cP u_g^*=\cP$ for every
$g\in\Gamma$. An example of such an algebra is $B\rtimes N$, where $B$
is a $\Gamma$-invariant von Neumann subalgebra of $Q$ and
$N\trianglelefteq\Gamma$ is a normal subgroup. We ask whether these
exhaust the $\Gamma$-invariant von Neumann subalgebras of
$Q\rtimes\Gamma$. In this paper, we study this problem for weakly
mixing trace-preserving actions, both in the situation described above and in its variant for
tensor products. When $Q=\Linf$ for a probability-measure-preserving (p.m.p.)\
action $\sigma$ of $\Gamma$ on a standard probability space $(X,\nu)$,
and $\tau_Q=\nu$, the question above asks whether the structure of an invariant subalgebra
can be recovered from a factor of the action on $(X,\nu)$ and a normal subgroup $N\trianglelefteq\Gamma$.

This question was already considered by Amrutam and Hartman~\cite{amrutam2024subalgebras} in the nonsingular
setting. They classified invariant intermediate algebras containing $\Linf$
when $(X,\nu)$ is a non-trivial factor of an ergodic nonsingular action
with the SH property~\cite[Theorem~1.2]{amrutam2024subalgebras}, and
conjectured a description of all invariant
subalgebras for Furstenberg--Poisson boundary actions of higher rank
lattices. In particular, they conjectured that every invariant subalgebra of
$L^\infty(G/P)\rtimes\Gamma$ is of the form
$L^\infty(G/Q)\rtimes\Lambda$, where $Q$ is parabolic and
$\Lambda\trianglelefteq\Gamma$~\cite[Conjecture, p.~381]{amrutam2024subalgebras}.
For the higher rank boundary actions considered there,
Zhou~\cite[Theorem~A]{zhou2026invariantpoisson} proved this
under the additional assumption that the invariant subalgebra in question has a nonscalar
intersection with either $L^\infty(G/P)$ or the group von Neumann algebra $L(\Gamma)$.
Jiang and Liu~\cite[Question~4.5]{JL} asked a related question for the
action of $\mathrm{PSL}_2(\Z)$ on $\T^2/\{\pm1\}$ induced by the
standard linear action. The present paper takes a first step toward
such a classification for weakly mixing p.m.p.~actions.

The group algebra case, corresponding to $X=\{*\}$ and $Q=\C1$, provides a model for this problem. A group
has the \emph{invariant subalgebras rigidity property}, or \emph{ISR
property}, if every invariant von Neumann subalgebra of $L(\Gamma)$
is $L(N)$ for some $N\trianglelefteq\Gamma$~\cite{amrutam2023invariant}.
\label{def:ISR}
Rigidity results for higher rank lattices were established by
Alekseev and Brugger~\cite{alekseev2021rigidity} and Kalantar and
Panagopoulos~\cite{kalantar2023invariant}.
Amrutam and Jiang~\cite{amrutam2023invariant} established it for
torsion-free non-amenable hyperbolic groups, and Chifan, Das and
Sun~\cite{chifan2023invariant} for ICC acylindrically hyperbolic groups.
Further results appear in
\cite{jiang2024example,dudko2024character,amrutam2025invariant,JL,jiang2026classification, amrutam2026relative, amrutam2026thompson}.
The corresponding problem for reduced group $C^*$-algebras was
studied in~\cite{kalantar2023invariant,amrutam2026invariant,amrutam2026psln}.

For crossed products, the dynamics introduces another source of
subalgebras. The theory of intermediate algebras shows how this
structure can be recovered when part of the crossed product is
already contained in the algebra. Subgroup correspondences for
crossed-product and fixed-point inclusions were established by
Choda~\cite{choda1978galois}, Izumi, Longo and
Popa~\cite{izumi1998galois}, and Cameron and
Smith~\cite{Cameron2015BimodulesIC,cameron2016intermediate}.
Their $C^*$-algebraic counterparts include the results of
Izumi~\cite{izumi2002inclusions} and Cameron and
Smith~\cite{cameron2019galois}.
Suzuki~\cite{Suz} recovered intermediate dynamical extensions from
intermediate crossed products of essentially free actions. For algebras
containing the group algebra, Amrutam~\cite{amrutam2021intermediate}
and Amrutam, Glasner and Glasner~\cite{amrutam2024reflecting} (also see \cite{amrutam2023intermediate}) gave
conditions under which all intermediate algebras arise from dynamical
factors. The latter authors also showed that such a description can
fail and described intermediate algebras which do not arise from
dynamical factors~\cite{amrutam2024nonabelian}.

An arbitrary invariant subalgebra need not contain either $Q$ or the group algebra. Thus the first question is
how to locate its Fourier coefficients in $Q$. Let $\mathbb{E}$ denote
the canonical conditional expectation from $Q\rtimes\Gamma$
onto $Q$. If $\cP=B\rtimes N$, then $\mathbb{E}(\cP)=B\subseteq \cP$.
Any classification of this form must therefore account for the
implication $\mathbb{E}(\cP)=\C1\Rightarrow\cP\subseteq L(\Gamma)$.
This is a difficult problem in its own right, and our first result proves it
under the pointwise weak mixing assumption.
The same argument applies to $L(\Gamma)\bar\otimes Q$ with the
diagonal action given by
$\theta_s(\lambda_t\otimes a)=\lambda_{sts^{-1}}\otimes\sigma_s(a)$
for $s,t\in\Gamma$ and $a\in Q$, where $\lambda_t$ denotes the canonical group unitary in $L(\Gamma)$. In this setting,
$\mathbb{E}$ denotes the expectation onto $\mathbb{C}1\bar{\otimes}Q$. The action $\sigma$ is called \emph{pointwise weakly mixing} if $\sigma_g$ is weakly mixing for every $g \in \Gamma$, $g \neq e$ (see Definition~\ref{def: relative weak mixing}).

\begin{thm}\label{thm:scalar-expectation}
Let $\Gamma$ be a countable discrete group and let
$\sigma$ be a pointwise weakly mixing trace-preserving action of $\Gamma$ on a tracial von Neumann algebra $(Q,\tau_Q)$. If
$\cP\subseteq Q\rtimes\Gamma$ is a $\Gamma$-invariant von Neumann
subalgebra and $\mathbb{E}(\cP)=\C1$, then
$\cP\subseteq L(\Gamma)$. The same conclusion holds for a subalgebra
$\cP\subseteq L(\Gamma)\bar\otimes Q$ invariant under the diagonal
action and satisfying $\mathbb{E}(\cP)=\C1$. In  both cases, if $\Gamma$ also has the
ISR property, then $\cP=L(N)$ for some $N\trianglelefteq\Gamma$.
\end{thm}
Surprisingly, we do not need the group to be ICC. The proof follows from a more general
commuting-square result below.

The condition $\mathbb{E}(\cP)\subseteq\cP$ is equivalent to
$\mathbb{E}\mathbb{E}_{\cP}=\mathbb{E}_{\cP}\mathbb{E}$,
where $\mathbb{E}_{\cP}$ is the trace-preserving expectation onto
$\cP$ (see Lemma~\ref{lem:commuting-square}). In other words,
$\cP$ and $Q$ form a \emph{commuting square} in the sense of \cite{Popa_csq_acta_math,Popa_csq_invent_math,GDJ_csq}.

The next theorem involves the notion of \emph{pointwise relative weak mixing}, which will be introduced in Section~\ref{sec:prelim}.

%In the following theorem, a (non-commutative) $\Gamma$-factor of $(Q,\tau_Q)$ means a $\Gamma$-invariant von Neumann subalgebra
%$B\subseteq Q$. It is proper if $B\neq Q$
%and non-trivial if $B\neq\C1$. When $Q=\Linf$, the $\Gamma$-factors are
%the subalgebras $B=L^\infty(Y,\eta)\subseteq\Linf$ given by the factors
%$(Y,\eta)$ of the action. Relative weak mixing over $\C1$ means
%ordinary weak mixing.

\begin{thm}\label{thm:cs-intro}
Let $\Gamma$ be a countable discrete group and let
$\sigma$ be a trace-preserving action of $\Gamma$ on a tracial von Neumann algebra $(Q,\tau_Q)$.
Let $\cP\subseteq Q\rtimes\Gamma$ be a $\Gamma$-invariant von
Neumann subalgebra such that $\mathbb{E}(\cP)\subseteq\cP$, and let
$B=\mathbb{E}(\cP)$. If the action  $\sigma$ is pointwise weakly mixing relative to $B$, then $\cP\subseteq B\rtimes\Gamma$.
The same holds for a von Neumann subalgebra of $L(\Gamma)\bar\otimes Q$
invariant under the diagonal action, where now 
$\mathbb{E}(\cP)=\mathbb{C}1\bar{\otimes} B$ and the conclusion is
$\cP\subseteq L(\Gamma)\bar\otimes B$.
\end{thm}

The scalar-expectation theorem is the case $B=\C1$.
In Section~\ref{sec:commuting-square} we prove an abstract version,
Theorem~\ref{thm:cs-engine}, and apply it to both crossed products and
tensor products. For crossed products, central $\Gamma$-freeness of the
inclusion $\cP\cap Q\subseteq Q$ (see Definition~\ref{def:central-free})
implies the commuting-square
condition by~\cite[proof of Theorem~4.6]{Suz} (see Lemma~\ref{lem:central-free-expectation}). This gives the following result. Here a (noncommutative) $\Gamma$-factor of $(Q,\tau_Q)$ means a $\Gamma$-invariant von Neumann subalgebra $B\subseteq Q$. It is proper if $B\neq Q$
and non-trivial if $B\neq\C1$. When $Q=\Linf$, the $\Gamma$-factors are the subalgebras $B=L^\infty(Y,\eta)\subseteq\Linf$ given by the factors $(Y,\eta)$ of the action in the usual dynamical sense.

\begin{samepage}
\begin{thm}\label{thm:main}
Let $\Gamma$ be a countable discrete group and let
$\sigma$ be a trace-preserving action of $\Gamma$ on a tracial von Neumann
algebra $(Q,\tau_Q)$ such
that
\begin{enumerate}
\item[\textup{(a)}] the action is pointwise weakly mixing relative to every proper $\Gamma$-factor of $(Q,\tau_Q)$, including
the trivial factor;
\item[\textup{(b)}] for every non-trivial $\Gamma$-factor $B$ of $(Q,\tau_Q)$,
the inclusion $B\subseteq Q$ is centrally $\Gamma$-free.
\end{enumerate}
Let $\cP\subseteq Q\rtimes\Gamma$ be a $\Gamma$-invariant von
Neumann subalgebra. Then
\begin{enumerate}
\item[\textup{(i)}] if $\cP\cap Q\neq\C1$, then
$\cP=(\cP\cap Q)\rtimes N$ for some $N\trianglelefteq\Gamma$;
\item[\textup{(ii)}] if $\mathbb{E}(\cP)=\C1$, then
$\cP\subseteq L(\Gamma)$; if moreover $\Gamma$ has the ISR property,
then $\cP=L(N)$ for some $N\trianglelefteq\Gamma$.
\end{enumerate}
\end{thm}
\end{samepage}
Again, note that we do not require the group $\Gamma$ to be ICC. However, if
$Q\neq\C1$, condition~\textup{(a)} forces $\Gamma$ to be
torsion-free (a periodic trace-preserving automorphism of a tracial
von Neumann algebra $Q\neq\C1$ cannot be weakly mixing).

When $Q=\Linf$ for a p.m.p.\ action on $(X,\nu)$, the inclusion of a
$\Gamma$-factor $B=L^\infty(Y,\eta)$ in $\Linf$ is \emph{centrally
$\Gamma$-free} if and only if the factor action
$\Gamma\curvearrowright(Y,\eta)$ is essentially free (see
Subsection~\ref{subsec:central-free}). Thus, in the commutative case,
condition~\textup{(b)} says that the action on every non-trivial factor
of $(X,\nu)$ is essentially free, and Theorem~\ref{thm:main} contains
the corresponding statement for p.m.p.\ actions.

The two assertions do not cover the case
$\cP\cap Q=\C1$ with $\mathbb{E}(\cP)\neq\C1$.
This case can occur under the hypotheses of the theorem. Indeed,
let an action of $\Z$ be induced by a mixing prime transformation of a
standard nonatomic probability space. Its only proper factor is
trivial, it is pointwise weakly mixing and has a null fixed-point
set. Thus \textup{(a)} and \textup{(b)} hold, the latter by the remark above.
Example~\ref{rem:nonicc-ideals}, motivated by
\cite[Theorem~5.12]{amrutam2024nonabelian}, gives an invariant
subalgebra with precisely these properties. Thus the hypotheses do
not imply that every invariant subalgebra is a canonically embedded
crossed product. Under \textup{(b)}, the missing case is exactly the
failure of the commuting-square condition; see
Remark~\ref{rem:open-case-cs}.

Without a commuting square, we require combinatorial control of the Fourier sums. For this purpose, we introduce the \emph{escape of conjuagte iterates (ECI) property} (see Definition~\ref{def:strongeci}); for ICC groups it is equivalent to the
\emph{unique roots property}, namely that $x^m=y^m$ implies $x=y$
for every $m\in\N$ and $x,y\in\Gamma$.
Theorem~\ref{thm:gen} proves
$\cP\subseteq\mathbb{E}(\cP)''\rtimes\Gamma$ when $\Gamma$ has the
ECI property and the action is pointwise weakly mixing relative
to $\mathbb{E}(\cP)''$.

We now turn our attention to classifying the invariant subalgebras of the tensor product von Neumann algebra $L(\Gamma)\bar\otimes\Linf$ equipped with the diagonal action. Earlier tensor-product results classify
invariant intermediate algebras containing an entire tensor factor;
see \cite{GeKadison1996,amrutam2025tensorcontinuous,amrutam2025productwstar}.
The theorem below concerns arbitrary invariant subalgebras. To our
knowledge, it is the first such classification in this diagonal,
weakly mixing p.m.p.\ setting. We denote by $Z(\cP)$ the centre of the algebra $\cP$.
\begin{thm}\label{thm:main-tensor}
Let $\Gamma$ be a countable group with trivial amenable radical and
the ISR property. Let
$\sigma$ be a p.m.p.~action of $\Gamma$ on $(X,\nu)$ which
is pointwise weakly
mixing relative to every proper $\Gamma$-factor, including the
trivial factor. If
$\cP\subseteq L(\Gamma)\bar\otimes\Linf$ is invariant under the
diagonal action, then $\cP=L(N)\bar\otimes B$ and
$Z(\cP)=\C1\bar\otimes B$
for some normal subgroup $N\trianglelefteq\Gamma$ and some
$\Gamma$-invariant von Neumann subalgebra $B\subseteq\Linf$.
\end{thm}
The groups $\F_n$, $2\leq n\leq\infty$, have trivial amenable
radical and ISR, and hence satisfy the group assumptions of
Theorem~\ref{thm:main-tensor}; see
\cite{amrutam2023invariant,chifan2023invariant}. More generally, the same holds for
torsion-free non-elementary hyperbolic groups.

Pointwise weak mixing itself does not suffice for these
classification results. Example~\ref{ex:relative-mixing-needed}
gives an invariant subalgebra of a Bernoulli crossed product whose
intersection with $\Linf$ is non-trivial but which is not a
canonically embedded crossed product.
Proposition~\ref{prop:direct-augmentation} gives the corresponding
example in the tensor product. In both cases, the required relative
weak mixing fails.

We also construct actions satisfying the hypotheses of our main results, using ideas from \cite[Section~17]{kechris2010global} and \cite{wrobel2026orbit}. Recall that
an action is \emph{prime} if its only factors are the trivial factor
and the whole space. % and \emph{totally weakly mixing} if its
%restriction to every infinite subgroup is weakly mixing. 
\begin{thm}
\label{thm:intro-prime}
For every $2\leq n\leq\infty$, there exists an essentially free, prime, pointwise
weakly mixing p.m.p.~action of $\F_n$. These actions may be
obtained by restricting a single action of $\F_2$ to an increasing
family of free subgroups.
\end{thm}

For a prime action, the only proper factor is trivial.
Theorem~\ref{thm:main}, applied to $Q=\Linf$ together with the remark following it, therefore gives the following corollary.

\begin{cor}\label{cor:intro-free}
Let $\sigma$ be an essentially free, prime,
pointwise weakly mixing p.m.p.~action of $\F_n$ on $(X,\nu)$, where
$2\leq n\leq\infty$. If $\cP\subseteq\Linf\rtimes\F_n$ is invariant
and $\cP\cap\Linf\neq\C1$, then $\cP=\Linf\rtimes N$ for some
$N\trianglelefteq\F_n$. If $\mathbb{E}(\cP)=\C1$, then
$\cP=L(N)$ for some $N\trianglelefteq\F_n$.
\end{cor}
\subsection{Organization of the paper}
Section~\ref{sec:prelim} collects the notation and auxiliary results,
including relative weak mixing, central freeness and the ECI property.
In Section~\ref{sec:commuting-square}, we prove
Theorem~\ref{thm:cs-engine} for tracial von Neumann algebras and deduce
the crossed-product and tensor-product containment results.
Section~\ref{sec:crossed} proves the crossed-product classification,
Theorem~\ref{thm:main}, by means of Theorem~\ref{thm:tracial-crossed-classification},
and the ECI theorem, Theorem~\ref{thm:gen}.
In Section~\ref{sec:tensor}, we first classify invariant subalgebras
under the centre-position assumption. The boundary argument in
Proposition~\ref{prop:insideamenradtensor} then gives
Theorem~\ref{thm:main-tensor}. Proposition~\ref{prop:direct-augmentation}
shows that pointwise mixing does not suffice for this classification.
Finally, Section~\ref{sec:prime} constructs prime and non-prime actions
satisfying the hypotheses of these results.

\section*{Acknowledgments} 
The first-named author thanks Prof.\ Stefaan Vaes for taking the time to discuss this project with him during his visit to the Department of Mathematics at KU Leuven in the first week of March 2025. We also thank Prof.\ Stefaan Vaes for pointing out an error in a previous proof of Theorem \ref{thm:gen} from June 2025.

Part of this work was carried out in June 2025 while we were attending the conference \emph{Ergodic Group Actions and Unitary Representations II} in B\k{e}dlewo. An earlier version of Theorem \ref{thm:cs-engine}, addressing crossed product algebras arising from p.m.p.\ actions of free groups, was obtained in February 2026 while we were attempting to generalize the character trick of \cite[Lemma 2.7]{dudko2024character}, which was used in the proof for algebraic actions.

TA is supported by National Science Centre, Poland Sonata, grant number
2025/59/D/ST1/03117. YJ is supported by National Natural Science Foundation of China, grant number 12471118. AS was partially supported by the National Science Center Grant OPUS-29 UMO-2025/57/B/ST1/00057.
\section*{Use of AI}
The mathematical ideas, statements, and proofs are the work of the human authors. ChatGPT 6.1 and Claude 5.5 Opus were used solely for typesetting and for correcting language and grammatical errors.
\section{Preliminaries}\label{sec:prelim}
Let $(Q,\tau_Q)$ be a von Neumann algebra with a faithful normal tracial
state, and let $\sigma$ be a trace-preserving action of a countable
discrete group $\Gamma$ on $Q$. The underlying tracial von Neumann algebras are assumed to have separable predual. All von Neumann subalgebras are assumed to be unital.
We write $u_g$, $g\in\Gamma$, for the canonical unitaries in the von Neumann algebraic crossed product
$Q\rtimes_\sigma\Gamma$ and $\mathbb{E}$ for the canonical conditional
expectation onto $Q$. When the action is clear, we omit $\sigma$ from the crossed-product notation. For $x\in Q\rtimes_\sigma\Gamma$, write
$x=\sum_{g\in\Gamma}f_gu_g$, where $f_g=\mathbb{E}(xu_g^*)\in Q$;
the sum converges in $L^2$-norm with respect to the canonical trace
$\tau=\tau_Q\circ\mathbb{E}$. We use the notation
$\langle a,b\rangle=\tau(b^*a)$ and $\|a\|_2=\tau(a^*a)^{1/2}$
whenever $a,b$ are elements in a tracial von Neumann algebra with a trace $\tau$. A von Neumann subalgebra
$\cP\subseteq Q\rtimes_\sigma\Gamma$ is \emph{$\Gamma$-invariant} if
$u_g\cP u_g^*=\cP$ for every $g\in\Gamma$.
For $s\in\Gamma$ and $a,b\in Q\rtimes_\sigma\Gamma$, we use
\begin{equation}\label{eq:basic}
\begin{aligned}
 \|u_sau_s^{-1}\|_2&=\|a\|_2,\\
 \|\mathbb{E}(ab)\|_1&\leq\|a\|_2\|b\|_2.
\end{aligned}
\end{equation}
The first identity follows from invariance of the trace, and the second
from the Cauchy--Schwarz inequality and the $L^1$-contractivity of the expectation.
For finite $F\Subset\Gamma$, let $x_F=\sum_{g\in F}f_gu_g$; then
$\|x_F\|_2\leq\|x\|_2$.
For a p.m.p.~action on $(X,\nu)$, we take $Q=\Linf$ and
$\tau_Q=\nu$. In particular, $\mathbb{E}$ then denotes the canonical
expectation from $\Linf\rtimes\Gamma$ onto $\Linf$.

We also consider the tensor product $L(\Gamma)\bar\otimes\Linf$ with the
\emph{diagonal action} given by
$\theta_s(\lambda_t\otimes f)=\lambda_{sts^{-1}}\otimes\sigma_s(f)$
for $s,t\in\Gamma$ and $f\in\Linf$, where $\lambda_t$ denotes the canonical group unitary in $L(\Gamma)$.
Thus a von Neumann subalgebra
$\cP\subseteq L(\Gamma)\bar\otimes \Linf$
is called $\Gamma$-invariant if
$\theta_s(\cP)=\cP$ for each $s\in\Gamma$. We will occasionally write $u_g$ for $\lambda_g \otimes 1_{\Linf} \in L(\Gamma) \bar\otimes \Linf$ ($g \in \Gamma$). The canonical trace $\tau$ on $L(\Gamma)\bar\otimes \Linf$ is given simply by $\tau_\Gamma \otimes \nu$, where $\tau_\Gamma$ is the canonical trace of $L(\Gamma)$. By analogy with the case above, here $\mathbb{E}$ denotes the expectation onto
$\mathbb{C}1\bar{\otimes}\Linf$.

In both the crossed-product and tensor-product settings, conditional expectations are always $\tau$-preserving. For a von Neumann subalgebra $\cP$ of a tracial von Neumann algebra, $\mathbb{E}_{\cP}$ denotes the trace-preserving conditional expectation onto $\cP$. In particular, in either setting, $\mathbb{E}_{L(\Gamma)}$ denotes the trace-preserving conditional expectation
onto the canonical group von Neumann algebra, identified with $L(\Gamma)\bar{\otimes}\mathbb{C}1$ in the tensor product.
It applies $\tau_Q$ to each Fourier coefficient in the crossed product, and $\nu$ to each coefficient in the tensor product.

Although we mainly consider tensor products for actions on commutative von Neumann algebras, the same conventions apply to a trace-preserving action $\sigma$ on an
arbitrary tracial von Neumann algebra $(Q,\tau_Q)$: we then consider
$L(\Gamma)\bar\otimes Q$ with the diagonal action
$\theta_s(\lambda_t\otimes a)=\lambda_{sts^{-1}}\otimes\sigma_s(a)$
($s,t\in\Gamma$, $a\in Q$), the trace $\tau_\Gamma\otimes\tau_Q$, the
expectation $\mathbb{E}$ onto $\mathbb{C}1\bar\otimes Q$, and the $L^2$ Fourier
expansion $x=\sum_{t\in\Gamma}\lambda_t\otimes f_t$, where
$1\otimes f_t=\mathbb{E}((\lambda_t^*\otimes1)x)$ with $f_t\in Q$. This
generality is used in Corollary~\ref{cor:cs-tensor} and
Theorem~\ref{thm:scalar-expectation}.

\subsection{Central freeness}\label{subsec:central-free}
Fix a free ultrafilter $\omega$ on $\N$.
For a tracial von Neumann algebra $(M,\tau)$, the \emph{tracial ultrapower}
$M^\omega$ is the quotient of the $C^*$-algebra $\ell^\infty(\N,M)$ of bounded
sequences in $M$ by the ideal
\[
\mathcal I_\omega=\Bigl\{(x_n)_n\in\ell^\infty(\N,M)\Bigm|\lim_{n\to\omega}\|x_n\|_2=0\Bigr\}.
\]
It is a tracial von Neumann algebra with the faithful normal trace
$\tau^\omega((x_n)_n)=\lim_{n\to\omega}\tau(x_n)$, and we identify $M$ with
its diagonal copy in $M^\omega$, consisting of the classes of constant
sequences. We shall use the following standard facts, for which we refer for example to \cite[Appendix A]{SinclairSmith}.
\begin{enumerate}[label=\textup{(\arabic*)}]
\item If $D\subseteq M$ is a von Neumann subalgebra, the classes of sequences
in $D$ form a von Neumann subalgebra $D^\omega\subseteq M^\omega$. In
particular, given a trace-preserving  action $\sigma$ of $\Gamma$ on a tracial von Neumann algebra $Q$ we have $B^\omega\subseteq Q^\omega\subseteq(Q\rtimes_\sigma\Gamma)^\omega$
for every von Neumann subalgebra $B\subseteq Q$. In this situation $B^\omega\cap Q'$ consists of
the elements of $B^\omega$ that commute with the diagonal copy of $Q$.
\item The trace-preserving conditional expectation $\mathbb{E}_M$ from
$M^\omega$ onto $M$ is given by
$\mathbb{E}_M((x_n)_n)=\mathrm{weak}^*\text{-}\lim_{n\to\omega}x_n$. Indeed,
for $x=(x_n)_n\in M^\omega$ and $y\in M$ we have
\[
\tau(\mathbb{E}_M(x)y)=\lim_{n\to\omega}\tau(x_ny)=\tau^\omega(xy).
\]
Since $D$ is weak$^*$-closed, it follows that $\mathbb{E}_M(D^\omega)=D$ for
every von Neumann subalgebra $D\subseteq M$.
\item Every trace-preserving automorphism $\alpha$ of $M$ induces an
automorphism $\alpha^\omega$ of $M^\omega$, given by
$\alpha^\omega((x_n)_n)=(\alpha(x_n))_n$. In particular,  a trace-preserving  action $\sigma$ of $\Gamma$ on a tracial von Neumann algebra $Q$ induces an
action $s\mapsto\sigma_s^\omega$ of $\Gamma$ on $Q^\omega$, and
$u_szu_s^*=\sigma_s^\omega(z)$ in $(Q\rtimes_\sigma\Gamma)^\omega$ for all
$z\in Q^\omega$ and $s\in\Gamma$.
\end{enumerate}
%None of this requires separability of the predual.

The following definition first appeared in \cite[Definition~4.3]{Suz} and plays the role of essential freeness in the noncommutative context (see the discussion below the definition).

\begin{definition}\label{def:central-free}
Consider a trace-preserving  action $\sigma$ of a countable group $\Gamma$ on a tracial von Neumann algebra $Q$.
A $\Gamma$-invariant inclusion $B\subseteq Q$ is \emph{centrally
$\Gamma$-free} if, for each $s\in \Gamma, s\ne e$, there are countably many
projections
$p_j\in B^\omega\cap Q'$ ($j \in \N$) such that $\sum_{j=1}^\infty p_j=1$ \textup{(}in the strong
operator topology\textup{)} and
$p_j\sigma_s^\omega(p_j)=0$ for every $j\in \N$.
%See, for example,
%\cite[Definition~4.3 and Remark~4.4\textup{(3)}]{Suz}.
\end{definition}

\emph{The commutative case.} When $Q=\Linf$, the algebra $Q^\omega$ is
abelian, so $B^\omega\cap Q'=B^\omega$. If $B=L^\infty(Y,\eta)$ corresponds
to a factor $(Y,\eta)$ of the action, then the inclusion $B\subseteq\Linf$ is
centrally $\Gamma$-free if and only if the factor action
$\Gamma\curvearrowright(Y,\eta)$ is essentially free.  Indeed, if the factor
action is essentially free and $s\ne e$, then \cite[Lemma~3.1]{Suz} gives a (finite)
partition $Y=\sqcup_{j=1}^nY_j$ into measurable sets with
$Y_i\cap sY_i=\varnothing$, and the projections $p_i=1_{Y_i}\in B\subseteq B^\omega$
are as required in Definition~\ref{def:central-free}. Conversely, suppose that
the fixed-point set $E\subseteq Y$ of some $s\ne e$ has positive measure.
Then $1_E\sigma_s(b)=1_Eb$ for every $b\in B$, hence $1_E\sigma_s^\omega(p)=1_Ep$
for every $p\in B^\omega$. For projections $p_j$ as in
Definition~\ref{def:central-free} we would get
$1_Ep_j=p_j1_E\sigma_s^\omega(p_j)=1_Ep_j\sigma_s^\omega(p_j)=0$ for every $j$,
and summing over $j\in \N$ would give $1_E=0$, a contradiction.

\begin{lem}\label{lem:central-free-expectation}
Suppose that $\sigma$ is a trace-preserving  action  of a countable group $\Gamma$ on a tracial von Neumann algebra $Q$. Let $B$ be a $\Gamma$-invariant subalgebra of $Q$ such that  $B\subseteq Q$ is centrally $\Gamma$-free. Then
$\mathbb{E}(x)\in W^*(B,x)$ for every $x\in Q\rtimes_\sigma\Gamma$,
and $B'\cap(B\rtimes_\sigma\Gamma)=Z(B)$. Here, $W^*(B, x)$ denotes the von Neumann subalgebra of $Q\rtimes_{\sigma}\Gamma$ generated by $B$ and $x$.
\end{lem}
\begin{proof}
For the first assertion, we follow the argument in
\cite[Theorem~4.6]{Suz}.

We shall write $M:=Q\rtimes_\sigma\Gamma$. Let $x\in M$, and consider $D=W^*(B,x)$.  We can write
$x=\sum_{t\in\Gamma}x_tu_t$, where
$x_t=\mathbb{E}(xu_t^*)\in Q$. Fix $s \in \Gamma, s\ne e$, and choose projections
$p_j\in B^\omega\cap Q'\subseteq D^\omega$ as in Definition~\ref{def:central-free}.
These projections are pairwise orthogonal. 
For $y\in M^\omega$, define $\Phi(y)=\sum_{j=1}^\infty p_jyp_j$.
The partial sums have norm at most $\|y\|$, and
\[
\sum_j\|p_jyp_j\|_2^2\le\sum_j\|p_jy\|_2^2=\|y\|_2^2.
\]
Since the summands are orthogonal in $L^2(M^\omega)$, the partial sums
are $\|\cdot\|_2$-Cauchy. Therefore $\Phi(y)\in M^\omega$, with
$\|\Phi(y)\|\le\|y\|$ and $\|\Phi(y)\|_2\le\|y\|_2$.
Let $\Psi=\mathbb{E}_M\circ\Phi$, with $\mathbb{E}_M$ as in fact~(2) above.
Then $\Psi$ is contractive in both the operator norm and the $L^2$-norm. For $y\in D$, the partial
sums defining $\Phi(y)$ belong to $D^\omega$ and have norm at most
$\|y\|$, so their $\|\cdot\|_2$-limit also belongs to $D^\omega$.
Since $\mathbb{E}_M(D^\omega)=D$, we get $\Psi(D)\subseteq D$.

We shall now compute the Fourier coefficients of $\Psi(y)$ for $y=\sum_{t\in \Gamma}y_tu_t\in M$.
For $t\in\Gamma$, fact~(3) gives $u_tp_j=\sigma_t^\omega(p_j)u_t$, so
$\Phi(y_tu_t)=c_tu_t$, where
$c_t=\sum_{j=1}^\infty p_jy_t\sigma_t^\omega(p_j)$.
The partial sums defining $c_t$ belong to $Q^\omega$, have norm at most
$\|y_t\|$, and converge in $\|\cdot\|_2$. Thus $c_t\in Q^\omega$.
Bimodularity gives $\Psi(y_tu_t)=\mathbb{E}_M(c_t)u_t$, where
$\mathbb{E}_M(c_t)\in Q$. By linearity and $\|\cdot\|_2$-contractivity,
$\Psi(y)=\sum_{t\in \Gamma}\mathbb{E}_M(c_t)u_t$ in $L^2(M)$.
Since the projections $p_j$ are pairwise orthogonal,
\[
\|\mathbb{E}_M(c_t)\|_2^2\le\|c_t\|_2^2
=\sum_{j=1}^\infty\|p_jy_t\sigma_t^\omega(p_j)\|_2^2
\le\sum_{j=1}^\infty\|p_jy_t\|_2^2=\|y_t\|_2^2.
\]
Since $p_j$ commutes with $y_t\in Q$, we have $c_e=y_e$ and
$c_s=y_s\sum_{j=1}^\infty p_j\sigma_s^\omega(p_j)=0$.
Thus $\Psi$ maps $D$ into $D$, fixes the Fourier coefficient at $e$,
annihilates the coefficient at $s$, and does not increase the
$\|\cdot\|_2$-norm of any Fourier coefficient. 

Now let $F\Subset\Gamma\setminus\{e\}$, write $F=\{s_1,\ldots,s_k\}$, and let
$\Psi_i$ be the map constructed above for $s=s_i$. Set $x^{(0)}=x$ and
$x^{(i)}=\Psi_i(x^{(i-1)})$ for $1\le i\le k$. Then $x^{(k)}\in D$ and
$\|x^{(k)}\|\le\|x\|$. Its Fourier coefficient at $e$ is $x_e$, its
coefficients at the elements of $F$ vanish, and its coefficient at any
$t$ has $\|\cdot\|_2$-norm at most $\|x_t\|_2$. Hence
\[
\|x^{(k)}-\mathbb{E}(x)\|_2^2\le\sum_{t\notin F\cup\{e\}}\|x_t\|_2^2.
\]
Taking an increasing sequence of finite sets $F$ exhausting
$\Gamma\setminus\{e\}$, we obtain a sequence in $D$, bounded in norm by
$\|x\|$, which converges to $\mathbb{E}(x)$ in $\|\cdot\|_2$.
Since the ball of radius $\|x\|$ in $D$ is $\|\cdot\|_2$-closed in $M$,
we conclude that $\mathbb{E}(x)\in D$.

For the second assertion, let $x=\sum_{s\in \Gamma}a_su_s\in B'\cap(B\rtimes_\sigma\Gamma)$,
where $a_s\in B$. Comparing the Fourier coefficients of $bx$ and $xb$
gives $ba_s=a_s\sigma_s(b)$ for every $b\in B$ and $s\in\Gamma$.
Fix $s\ne e$ and choose the projections $p_j$ as in
Definition~\ref{def:central-free}. The same identity in the ultrapower,
applied to representing sequences in $B$,
gives $p_ja_s=a_s\sigma_s^\omega(p_j)$  for every $j \in \N$. Since $p_j$ commutes with
$a_s$, multiplication on the right by $p_j$ gives $a_sp_j=a_s\sigma_s^\omega(p_j)p_j
=a_s(p_j\sigma_s^\omega(p_j))^*=0$.
Summing over $j$ yields $a_s=0$. Thus $x=a_e\in Z(B)$.
The reverse inclusion is immediate.
\end{proof}

\subsection{A covariant representation for compact dynamical systems}\label{subsec:boundary-rep}
Let $\Gamma$ be a countable discrete group and let $(Q,\tau_Q)$ be a von Neumann
algebra with a faithful normal tracial state. Suppose that $\sigma$ is a
trace-preserving action of $\Gamma$ on $Q$. Let $M=L(\Gamma)\bar\otimes Q$ carry the diagonal action
$\theta_s(\lambda_t\otimes a)=\lambda_{sts^{-1}}\otimes\sigma_s(a)$, $s, t \in \Gamma$, and let
$\widetilde{M}=M\rtimes_\theta\Gamma$, with $v_s$ the canonical unitaries
implementing $\theta$ and $u_t=\lambda_t\otimes1\in M$. We use the standard representation of
$\widetilde{M}$ on
$L^2(\widetilde{M})\cong \ell^2(\Gamma)_r\otimes L^2(Q)\otimes\ell^2(\Gamma)_h$, so that, %given for $x\in M$
%by $(\pi(x)\xi)(h)=\theta_{h^{-1}}(x)\xi(h)$ and by $(v_s\xi)(h)=\xi(s^{-1}h)$. Thus, 
for $a\in Q$ (identified with $1\otimes a\in M$), $s,t,r,h\in\Gamma$ and $\eta\in L^2(Q)$, the action on elementary tensors is given by
\begin{align*}
\pi(a)(\delta_r\otimes\eta\otimes\delta_h)
&=\delta_r\otimes\sigma_{h^{-1}}(a)\eta\otimes\delta_h,\\
\pi(u_t)(\delta_r\otimes\eta\otimes\delta_h)
&=\delta_{h^{-1}th r}\otimes\eta\otimes\delta_h,\\
v_s(\delta_r\otimes\eta\otimes\delta_h)
&=\delta_r\otimes\eta\otimes\delta_{sh}.
\end{align*}
We suppress $\pi$ from the notation. Let $Z$ be a compact Hausdorff
$\Gamma$-space with a point $x_0\in Z$ whose orbit is dense. For
$f\in C(Z)$, define a multiplication operator $m_f$ on $L^2(\widetilde{M})$ by
$m_f(\delta_r\otimes\eta\otimes\delta_h) = f(hrx_0)\delta_r\otimes\eta\otimes\delta_h$, with $\eta, r, h$ as above.
The density of $\Gamma x_0$ makes this representation of $C(Z)$ faithful.
The next lemma shows that it is covariant for both copies of $\Gamma$.

\begin{lem}\label{lem:boundary-covariance}
Write $(g\cdot f)(x)=f(g^{-1}x)$ for $g\in\Gamma$, $f\in C(Z)$ and
$x\in Z$. Then
$v_sm_fv_s^*=m_{s\cdot f}$ and $u_tm_fu_t^*=m_{t\cdot f}$ for $s,t\in\Gamma$,
equivalently $u_t^*m_f=m_{t^{-1}\cdot f}u_t^*$ with $(t^{-1}\cdot f)(x)=f(tx)$. Moreover,
$C(Z)$ commutes with $\mathbb{C}1\bar{\otimes}Q$ inside $\mathbb B(L^2(\widetilde{M}))$.
\end{lem}
\begin{proof}
For the first relation we have
\begin{align*}
v_s m_f v_s^*(\delta_r\otimes\eta\otimes\delta_h)
&=v_s m_f(\delta_r\otimes\eta\otimes\delta_{s^{-1}h})\\
&=f(s^{-1}hrx_0)\delta_r\otimes\eta\otimes\delta_h\\
&=(s\cdot f)(hrx_0)\delta_r\otimes\eta\otimes\delta_h
=m_{s\cdot f}(\delta_r\otimes\eta\otimes\delta_h),
\end{align*}
and for the second,
\begin{align*}
u_t m_f u_t^*(\delta_r\otimes\eta\otimes\delta_h)
&=u_t m_f(\delta_{h^{-1}t^{-1}hr}\otimes\eta\otimes\delta_h)\\
&=u_t\bigl(f(t^{-1}hrx_0)\delta_{h^{-1}t^{-1}hr}\otimes\eta\otimes\delta_h\bigr)\\
&=f(t^{-1}hrx_0)\delta_r\otimes\eta\otimes\delta_h\\
&=(t\cdot f)(hrx_0)\delta_r\otimes\eta\otimes\delta_h
=m_{t\cdot f}(\delta_r\otimes\eta\otimes\delta_h).
\end{align*}
The final assertion holds because both $C(Z)$ and $\mathbb{C}1\bar{\otimes}Q$ act diagonally in the
two $\ell^2(\Gamma)$-coordinates, the former by scalars on each fibre.
\end{proof}

\subsection{Weak mixing}\label{subsec:wm}
For a trace-preserving automorphism $\alpha$ of $(Q,\tau_Q)$ and an
$\alpha$-invariant von Neumann subalgebra $B\subseteq Q$, the expectation
$\mathbb{E}_B$ commutes with $\alpha$. When $Q=\Linf$ and
$T\in\Aut(X,\nu)$, we write $T\xi=\xi\circ T^{-1}$ for
$\xi\in\Linf$, and apply the definitions below to this automorphism
of $\Linf$.
\begin{samepage}
\begin{definition}\label{def: relative weak mixing}
Let $\alpha$ be a trace-preserving automorphism of $(Q,\tau_Q)$ and let
$B\subseteq Q$ be an $\alpha$-invariant von Neumann subalgebra.
The automorphism $\alpha$ is \emph{weakly mixing relative to $B$} if
$\frac1m\sum_{n=0}^{m-1}\norm{\mathbb{E}_B(\zeta^*\alpha^n(\xi))}_2^2\to0$ as $m\to\infty$
for all $\xi,\zeta\in Q$ with
$\mathbb{E}_B(\xi)=\mathbb{E}_B(\zeta)=0$.
For $B=\C1$, this is \emph{weak mixing}.
A trace-preserving action $\sigma$ of $\Gamma$ on $(Q,\tau_Q)$ is
\emph{pointwise weakly mixing relative to $B$} if $B$ is
$\Gamma$-invariant and $\sigma_g$ is weakly mixing relative to $B$
for every $g\ne e$. For $B=\C1$, we say that $\sigma$ is
\emph{pointwise weakly mixing}.
\end{definition}
\end{samepage}

For a p.m.p.\ action $\sigma$ of $\Gamma$ on $(X,\nu)$, weak mixing
in the sense of Definition~\ref{def: relative weak mixing} is the classical
notion: $\sigma_g$ is weakly mixing if and only if the transformation of
$(X,\nu)$ given by $g$ is weakly mixing, that is, $g$ acts ergodically on
$(X\times X,\nu\times\nu)$ diagonally, or equivalently, its Koopman operator
has no eigenvectors in $L^2(X,\nu)\ominus\C1$; see e.g.\
\cite{BergelsonGorodnik}. Thus a p.m.p.\ action is pointwise weakly mixing
if and only if every $g\neq e$ acts by a weakly mixing transformation; for
actions of free groups this is total weak mixing in the sense of
\cite{kechris2010global} (see Section~\ref{sec:prime}).
Remark~\ref{remark: relative weak mixing for diagonal action wrt a coordinate factor}
below gives the corresponding description of weak mixing relative to a factor.

If $\mathbb{E}_B(\xi)=0$, the same convergence holds for arbitrary
$\zeta\in Q$. Indeed, writing $\zeta_0=\zeta-\mathbb{E}_B(\zeta)$,
bimodularity and equivariance give
\[
 \mathbb{E}_B(\zeta^*\alpha^n(\xi))
 =\mathbb{E}_B(\zeta_0^*\alpha^n(\xi))
   +\mathbb{E}_B(\zeta)^*\alpha^n(\mathbb{E}_B(\xi))
 =\mathbb{E}_B(\zeta_0^*\alpha^n(\xi)).
\]

\begin{remark}\label{remark: relative weak mixing for diagonal action wrt a coordinate factor}
    For an ergodic $T \in \Aut(X,\nu)$, relative weak mixing is equivalent to ergodicity of the diagonal automorphism $T\times T\in \Aut(X\times_YX,\nu\times_{\mu}\nu)$, where $X\times_YX=\{(x_1,x_2)\in X\times X\mid\pi(x_1)=\pi(x_2)\}$ denotes the fibred product space and $\pi$ denotes the factor map from $(X,\nu)$ to $(Y,\mu)$ (see e.g.\ \cite{Furman_imrn}). For a reference, see, e.g.,\ \cite[Proposition~7.30 in Chapter~7]{EMW_book} or \cite[Theorem~7.5]{Furstenberg_1977}. In particular, a diagonal automorphism of two weakly mixing automorphisms, say $T\times S\in\Aut(X\times Z, \nu\times \mu)$, is weakly mixing relative to $L^{\infty}(X, \nu)\bar{\otimes}\mathbb{C}1_{L^{\infty}(Z,\mu)}$. Indeed, this follows since the automorphism $ (T\times S)\times (T\times S)\in \Aut((X\times Z)\times_X(X\times Z), (\nu\times \mu)\times_{\nu}(\nu\times \mu))$ is conjugate to  $T\times S\times S\in\Aut(X\times Z\times Z, \nu\times \mu\times \mu)$, which is ergodic. For further discussion, see \cite{BergelsonGorodnik}.
\end{remark}

Whenever we say that ``$T$ is weakly mixing relative to the factor $Y$'', we mean that $T$
is weakly mixing relative to $B=L^\infty(Y)$ in the sense above; being weakly mixing relative to the trivial factor means simply weakly mixing.

The next lemma gives a sequence along which the relative correlations tend to zero
for all pairs of distinct times. %For $B=\C1$ it applies to weak mixing.

\begin{samepage}
\begin{lem}\label{lem:relative-pairwise-wm-selection}
Let $\alpha$ be a trace-preserving automorphism of $(Q,\tau_Q)$ which
is weakly mixing relative to an $\alpha$-invariant von Neumann subalgebra
$B\subseteq Q$, and let $\Xi\Subset Q$ be finite. Then there exists a strictly increasing sequence
$(m_k)_{k\ge1}\subseteq\N$ such that
$\|\mathbb{E}_B(\eta^*\alpha^{m_\ell-m_k}(\xi))\|_2\to0$ as
$\max\{k,\ell\}\to\infty$ with $k\neq\ell$, for all
$\xi,\eta\in\Xi$ satisfying $\mathbb{E}_B(\xi)=\mathbb{E}_B(\eta)=0$.
\end{lem}
\end{samepage}

\begin{proof}
Let $\Xi_0=\{\xi\in\Xi\mid\mathbb{E}_B\xi=0\}$. If $\Xi_0$ is empty, any
strictly increasing sequence works. Otherwise, for $\xi,\eta\in\Xi_0$, set
$q_{\xi,\eta}(n)=\|\mathbb{E}_B(\eta^*\alpha^n(\xi))\|_2$.
For each positive integer $r$, relative weak mixing implies that
$\{n\ge1\mid q_{\xi,\eta}(n)\ge1/r\}$ has density zero, since its proportion
in $\{1,\ldots,N\}$ is bounded by
$r^2N^{-1}\sum_{n=1}^Nq_{\xi,\eta}(n)^2$.
Equivariance, preservation of adjoints and invariance of the $L^2$-norm give
$q_{\xi,\eta}(-n)=q_{\eta,\xi}(n)$. Hence
$S_r=\{n\ge1\mid\max_{\xi,\eta\in\Xi_0}\max_{\varepsilon\in\{1,-1\}} q_{\xi,\eta}(\varepsilon n)<1/r\}$
has density one. Choose $m_1\ge1$ and, recursively, choose
$m_k>m_{k-1}$ in $\bigcap_{j<k}(m_j+S_k)$; this finite intersection has
density one. If $\ell>k$, then $m_\ell-m_k\in S_\ell$; if $k>\ell$,
use the negative-time inequality from $S_k$. The required norm is thus
less than $1/\max(k,\ell)$ for every unequal pair of indices.
\end{proof}

The following two lemmas turn pairwise asymptotic orthogonality into weak convergence.

\begin{lem}\label{lem:weak-zero-from-pairwise}
Let $\widehat{\mathcal H}$ be a Hilbert space and let $(\eta_k)_{k\in\N}\subseteq\widehat{\mathcal H}$ be a bounded sequence. Assume that $\langle\eta_k,\eta_\ell\rangle\to0$ as
$\max\{k,\ell\}\to\infty$, $k\neq\ell$. Then $\eta_k\to0$ weakly as $k$ tends to infinity.
\end{lem}
\begin{proof}
As closed balls in $\widehat{\mathcal H}$ are weakly sequentially compact, it suffices to show that every weak cluster point of $(\eta_k)_{k \in \N}$ equals $0$. Let $\eta\in\widehat{\mathcal H}$ and suppose that $\eta_{k_\ell}\to\eta$
weakly as $\ell$ tends to infinity. Fix $j\in \N$. Since $\langle\eta_{k_j},\eta_{k_\ell}\rangle\to0$ as $\ell\to\infty$, we see that $\langle\eta_{k_j},\eta\rangle=0$. Letting $j\to\infty$ and
using again the assumed weak convergence gives $\|\eta\|^2=0$. Thus $\eta=0$.
\end{proof}

\begin{lem}\label{lem:stinespring-l2-small}
Let $(N,\tau)$ be a tracial von Neumann algebra, let $Q\subseteq N$ be a von Neumann subalgebra, and
let $\mathbb{E}_Q$ be the trace-preserving conditional expectation from $N$ onto $Q$. Let $(a_k)_{k\in\N}$ and $(b_k)_{k\in\N}$
be bounded sequences of elements in $N$. Assume that
$\|\mathbb{E}_Q(a_ka_\ell^*)\|_2\to0$ as $\max\{k,\ell\}\to\infty$, $k\neq\ell$. Then
$b_k\mathbb{E}_Q(a_k)\to0$ in the weak operator topology as $k$ tends to infinity.
\end{lem}
\begin{proof}
Let $(\pi,\widehat{\cH},V)$ be a Stinespring dilation for $\mathbb{E}_Q$, viewed as a u.c.p.\ map on
$N\subseteq \mathbb B(L^2(N, \tau))$, so that $\mathbb{E}_Q(x)=V^*\pi(x)V$ for $x\in N$ (\cite[Theorem~1.5.3]{BroOza08}). For each $k \in \N$ set
$T_k=b_kV^*\pi(a_k)\in \mathbb B(\widehat{\cH};L^2(N, \tau))$. Then for $k, \ell \in \N$ we have
$T_kT_\ell^*=b_k\mathbb{E}_Q(a_ka_\ell^*)b_\ell^*$.
Fix $\eta\in N$. Setting $z_{k,\ell}=\mathbb{E}_Q (a_ka_\ell^*)$ we have
\[
|\langle T_k^*\eta,T_\ell^*\eta\rangle|=|\tau(\eta^*b_{\ell}z_{\ell,k}b_k^*\eta)|\le
\|b_k^*\eta\eta^*b_{\ell}\|_2\|z_{\ell,k}\|_2.
\]
Since $(b_k)_{k\in\N}$ is bounded, the displayed expression tends to $0$
as $\max\{k,\ell\}\to\infty$, $k\neq\ell$. By approximation and the uniform bound on
$(T_k)_{k\in \N}$, the same holds for every $\eta\in L^2(N,\tau)$. Lemma~\ref{lem:weak-zero-from-pairwise} gives
$T_k^*\eta\to0$ weakly for every $\eta\in L^2(N, \tau)$. Thus also $b_k\mathbb{E}_Q(a_k)=T_kV\to0$ in the weak operator topology.
\end{proof}

\subsection{Conjugation and centralizers}\label{subsec:eci-prelim}
The \emph{ECI property} introduced below says that conjugation by high powers of a fixed element
eventually moves any finite set outside the centralizer off any other finite set.
We shall show that for ICC groups it coincides with the \emph{unique roots property}, which we recall next.
For $g\in\Gamma$, write $C_\Gamma(g)=\{h\in\Gamma\mid hg=gh\}$ for its centralizer.

{

\begin{definition}A group $\Gamma$ satisfies the unique roots property if for every $g\in \Gamma$ and $n\in\mathbb N$ there exists at most one $f\in \Gamma$ such that $f^n=g$.
\end{definition}}

Examples of groups with the unique roots property include free groups and, more generally, torsion-free hyperbolic groups \cite[Lemma~2.2]{BB}, as well as diagram groups (including Thompson's group $F$) \cite[Corollary 2.26]{genevois2025diagram}.

\begin{definition}
\label{def:strongeci}
A group $\Gamma$ satisfies the ECI property if for every
$e\neq g\in \Gamma$ and every pair of
finite sets $F\Subset \Gamma\setminus C_\Gamma(g)$ and $R \Subset \Gamma$ there exists $N \in \N$ such that
$g^{n}Fg^{-n}\cap R=\emptyset$ for all $n >N$.
\end{definition}

{

\begin{lem}\label{lemma: unique roots -> ECI}
Let $\Gamma$ be a countable ICC group. Then it satisfies the ECI property if and only if it satisfies    the unique  roots property.
\end{lem}

\begin{proof}
We may assume that $\Gamma$ is non-trivial.

Assume that $\Gamma$  satisfies    the unique roots property.    Fix $g\in \Gamma\setminus \{e\}$ and let $f\in \Gamma\setminus C_\Gamma(g)$. Assume that there exist $n,m\in\mathbb N$, $n<m$, such that $g^nfg^{-n}=g^mfg^{-m}$. Then we have
    $(f^{-1}gf)^{m-n}=f^{-1}g^{m-n}f=g^{m-n}$. The unique roots property implies that $f^{-1}gf=g$, which contradicts the fact that $f$ does not belong to the centralizer of $g$. Thus, the elements $g^nfg^{-n}$, $n\in\mathbb N$, are pairwise distinct.

   Now let $F\Subset \Gamma\setminus C_\Gamma(g)$ and $R \Subset \Gamma$ be finite sets.
For each $f\in F$ and $r\in R$, there is at most one $n\in\N$ such that $g^nfg^{-n}=r$.
Thus there are only finitely many such integers. Choose $N\geq1$ larger than all of them,
taking $N=1$ if there are none. Then $g^nFg^{-n}\cap R=\emptyset$ for every $n>N$, proving ECI.

     Assume then that $\Gamma$ has the ECI property and does not have the unique  roots property. Then there exist $g\neq h\in \Gamma$ and $n\in\mathbb N$ such that $g^n=h^n$. Assume first $g,h\neq e$. If $h\notin C_\Gamma(g)$, set $F=\{g^ihg^{-i}\mid i=0,\ldots,n-1\}$. Clearly, $F\subseteq \Gamma\setminus C_\Gamma(g)$, but $gFg^{-1}=F$, which contradicts the ECI property. Let then $h\in C_\Gamma(g)$. Then $(h^{-1}g)^n=e$. Thus, we arrive at the case $g^n=e$ for some $g\neq e$. As $\Gamma$ is assumed to be ICC, $g$ cannot belong to the centre of $\Gamma$. Take $h$ not commuting with $g$ and set $F$ as above. We then again obtain a contradiction to the ECI property, which finishes the proof.
\end{proof}
}

\section{The commuting square}\label{sec:commuting-square}
In this section, we isolate the mechanism behind all our classification results. Let $(M,\tau)$ be a
tracial von Neumann algebra and let $A,\cP\subseteq M$ be von Neumann subalgebras, with
trace-preserving conditional expectations $\mathbb{E}_A$ and $\mathbb{E}_{\cP}$. Following \cite{Popa_csq_invent_math,Popa_csq_acta_math,GDJ_csq}, we say that $\cP$ and
$A$ form a \emph{commuting square} if $\mathbb{E}_A\mathbb{E}_{\cP}=\mathbb{E}_{\cP}\mathbb{E}_A$. The following characterisation is well known; see, for example, \cite[Proposition 9.51]{EvansKawahigashi}.
\begin{lem}\label{lem:commuting-square}
Let $(M,\tau)$ be a tracial von Neumann algebra, let $A,\cP\subseteq M$ be von Neumann
subalgebras, and let $B=\cP\cap A$. The following conditions are equivalent.
\begin{enumerate}
\item[\textup{(i)}] $\mathbb{E}_A(\cP)\subseteq\cP$;
\item[\textup{(ii)}] $\mathbb{E}_{\cP}(A)\subseteq A$;
\item[\textup{(iii)}] $\mathbb{E}_A\mathbb{E}_{\cP}=\mathbb{E}_{\cP}\mathbb{E}_A$.
\end{enumerate}
In this case $\mathbb{E}_A\mathbb{E}_{\cP}=\mathbb{E}_{\cP}\mathbb{E}_A=\mathbb{E}_B$. In particular,
$\mathbb{E}_A(\cP)=B=\mathbb{E}_{\cP}(A)$ and $\mathbb{E}_{\cP}(a)=\mathbb{E}_B(a)$ for every $a\in A$,
and the diagram
\[
\begin{tikzcd}
M \arrow[r,"\mathbb{E}_{\cP}"] \arrow[d,"\mathbb{E}_A"'] & \cP \arrow[d,"\mathbb{E}_A"]\\
A \arrow[r,"\mathbb{E}_B"'] & B=\cP\cap A
\end{tikzcd}
\]
commutes in the sense that $\mathbb{E}_A\circ\mathbb{E}_{\cP}=\mathbb{E}_B\circ\mathbb{E}_A$.
\end{lem}
\begin{comment}\begin{proof}
Assume (i). Since $\mathbb{E}_A(\cP)\subseteq\cP\cap A=B$ and $x=\mathbb{E}_A(x)$ for every $x\in B$, we
have $\mathbb{E}_A(\cP)=B$. We claim that $\mathbb{E}_{\cP}(a)=\mathbb{E}_B(a)$ for every $a\in A$.
Indeed, let $a\in A$ and $m\in\cP$. Since $\mathbb{E}_A(m)\in B$, we get
\[
\tau\big(\mathbb{E}_{\cP}(a)m\big)=\tau(am)=\tau\big(a\mathbb{E}_A(m)\big)
=\tau\big(\mathbb{E}_B(a)\mathbb{E}_A(m)\big)=\tau\big(\mathbb{E}_B(a)m\big).
\]
As $\mathbb{E}_B(a)\in B\subseteq\cP$, taking $m=(\mathbb{E}_{\cP}(a)-\mathbb{E}_B(a))^*$ and using
faithfulness of $\tau$ proves the claim. Consequently,
$\mathbb{E}_{\cP}\mathbb{E}_A=\mathbb{E}_B\mathbb{E}_A=\mathbb{E}_B$. Viewing the three expectations as
orthogonal projections on $L^2(M,\tau)$ and taking adjoints, we also get
$\mathbb{E}_A\mathbb{E}_{\cP}=\mathbb{E}_B$. This proves (iii) together with the remaining assertions.
Since the statement is symmetric in $A$ and $\cP$, (ii) implies (iii) in the same way. Finally, if
(iii) holds, then $\mathbb{E}_A(\cP)=\mathbb{E}_A\mathbb{E}_{\cP}(\cP)=\mathbb{E}_{\cP}\mathbb{E}_A(\cP)
\subseteq\cP$, and similarly $\mathbb{E}_{\cP}(A)\subseteq A$.
\end{proof}
\end{comment}

\begin{remark}\label{rem:cs-crossed}
Let $\sigma$ be a trace-preserving action of $\Gamma$ on a tracial von Neumann algebra $(Q,\tau_Q)$ and let
$\cP\subseteq Q\rtimes_\sigma\Gamma$ be a von Neumann subalgebra. For $A=Q$ we have
$\mathbb{E}_A=\mathbb{E}$, so $\cP$ and $Q$ form a commuting square if and only if
$\mathbb{E}(\cP)\subseteq\cP$. In this case $\mathbb{E}(\cP)=\cP\cap Q$ and
$\mathbb{E}_{\cP}(f)=\mathbb{E}_{\mathbb{E}(\cP)}(f)$ for every $f\in Q$. If, moreover, $\cP$ is
$\Gamma$-invariant, then $\mathbb{E}(\cP)$ is a $\Gamma$-invariant von Neumann subalgebra of $Q$.
When $(Q,\tau_Q)=(\Linf,\nu)$, this means that $\mathbb{E}(\cP)=L^\infty(Y,\eta)$ for a $\Gamma$-equivariant factor map
from $(X,\nu)$ to $(Y,\eta)$. In the scalar case, Lemma~\ref{lem:commuting-square} shows that $\mathbb{E}(\cP)=\C1$ if and only if $\mathbb{E}_{\cP}(f)=\tau_Q(f)1$ for every
$f\in Q$. Every
algebra of the form $B\rtimes N$ forms a commuting square with $Q$. For a p.m.p.~action, the same remarks apply to
$L(\Gamma)\bar\otimes\Linf$, with $A=\mathbb{C}1\bar{\otimes}\Linf$.
\end{remark}

The next theorem gives the vanishing statement used in both the crossed product and the tensor
product. It applies relative weak mixing to elements orthogonal to the intersection of two invariant subalgebras.

\begin{thm}\label{thm:cs-engine}
Let $(M,\tau)$ be a tracial von Neumann algebra, let $\theta$ be a trace-preserving action of $\Gamma$ on $M$, and let $A\subseteq M$ be a $\theta$-invariant von Neumann subalgebra.
Write $\sigma=\theta|_A$. Let $\cP\subseteq M$ be a
$\theta$-invariant von Neumann subalgebra which forms a commuting square with $A$, and let
$B=\cP\cap A$. Let $g\in\Gamma$ and let $w\in M$ be a unitary with $\theta_g(w)=w$. If $\sigma_g$
is weakly mixing relative to $B$, then $\mathbb{E}_{\cP}(fw)=0$ for every
$f\in A$ with $\mathbb{E}_B(f)=0$.
\end{thm}
\begin{proof}
Since $\cP$ is $\theta$-invariant and $\theta$ preserves $\tau$, the expectation $\mathbb{E}_{\cP}$
is $\theta$-equivariant; also $B$ is $\sigma$-invariant, so $\mathbb{E}_B$ commutes with $\sigma_g$.
By Lemma~\ref{lem:commuting-square}, $\mathbb{E}_{\cP}(a)=\mathbb{E}_B(a)$ for every $a\in A$.

Fix $f\in A$ with $\mathbb{E}_B(f)=0$; then also $\mathbb{E}_B(f^*)=0$. Apply
Lemma~\ref{lem:relative-pairwise-wm-selection} to $\alpha=\sigma_g$ and $\Xi=\{f^*\}$, and let
$(m_k)_{k\geq1}$ be the resulting sequence. Then
$\|\mathbb{E}_B(f\sigma_g^{m_\ell-m_k}(f^*))\|_2\to0$
as $\max\{k,\ell\}\to\infty$ with $k\neq\ell$.
Let $a_k=\theta_g^{m_k}(fw)=\sigma_g^{m_k}(f)w$ and
$b_k=\mathbb{E}_{\cP}(a_k)^*=\theta_g^{m_k}(\mathbb{E}_{\cP}(fw))^*$.
Both sequences are bounded by $\|f\|_\infty$. Since $w$ is a unitary, we have
$a_ka_\ell^*=\sigma_g^{m_k}(f)\sigma_g^{m_\ell}(f)^* =\sigma_g^{m_k}\big(f\sigma_g^{m_\ell-m_k}(f^*)\big)\in A$.
Consequently,
\[
\|\mathbb{E}_{\cP}(a_ka_\ell^*)\|_2=\|\mathbb{E}_B(a_ka_\ell^*)\|_2
=\big\|\mathbb{E}_B\big(f\sigma_g^{m_\ell-m_k}(f^*)\big)\big\|_2\longrightarrow0
\]
as $\max\{k,\ell\}\to\infty$, $k\neq\ell$. Lemma~\ref{lem:stinespring-l2-small} therefore gives
$b_k\mathbb{E}_{\cP}(a_k)\to0$ in the weak operator topology. Pairing with the trace vector, we obtain
\[
\|\mathbb{E}_{\cP}(fw)\|_2^2
=\tau\Big(\theta_g^{m_k}\big(\mathbb{E}_{\cP}(fw)^*\mathbb{E}_{\cP}(fw)\big)\Big)
=\tau\big(b_k\mathbb{E}_{\cP}(a_k)\big)\longrightarrow0 .
\]
Hence $\mathbb{E}_{\cP}(fw)=0$.
\end{proof}

We now specialise Theorem~\ref{thm:cs-engine} to the two ambient algebras of this paper. The first
corollary gives the crossed product statement of Theorem~\ref{thm:cs-intro} from the introduction.

\begin{cor}\label{cor:cs-crossed}
Let $\sigma$ be a trace-preserving action of $\Gamma$ on a tracial von Neumann algebra $(Q,\tau_Q)$ and let
$\cP\subseteq Q\rtimes_\sigma\Gamma$ be a $\Gamma$-invariant von Neumann subalgebra such that
$\mathbb{E}(\cP)\subseteq\cP$. Write $B=\mathbb{E}(\cP)=\cP\cap Q$ as in
Remark~\ref{rem:cs-crossed}. If $\sigma$ is pointwise weakly mixing relative to $B$, then
$\cP\subseteq B\rtimes\Gamma$.
\end{cor}
\begin{proof}
To establish the claim, it is enough to show that $\mathbb{E}_{\cP}(fu_s)=0$ whenever $s\in\Gamma$
and $f\in Q$ satisfies $\mathbb{E}_B(f)=0$. Assuming this, let
$x\in\cP$ have Fourier coefficients $f_s=\mathbb{E}(xu_s^*)$, $s\in\Gamma$. For every $s\in\Gamma$
and every $f\in Q$ with $\mathbb{E}_B(f)=0$ we have
$\langle f_s,f\rangle=\tau\big((fu_s)^*x\big)=\langle x,\mathbb{E}_{\cP}(fu_s)\rangle=0$.
Since such $f$ are dense in $L^2(Q,\tau_Q)\ominus L^2(B,\tau_Q)$, it follows that
$f_s\in L^2(B,\tau_Q)\cap Q=B$ for every $s\in\Gamma$, that is, $x\in B\rtimes\Gamma$.

It remains to prove that $\mathbb{E}_{\cP}(fu_s)=0$ whenever $\mathbb{E}_B(f)=0$. For $s=e$,
this follows from Remark~\ref{rem:cs-crossed}, since $\mathbb{E}_{\cP}(f)=\mathbb{E}_B(f)=0$. For
$s\neq e$, we apply Theorem~\ref{thm:cs-engine} with $M=Q\rtimes_\sigma\Gamma$,
$\theta_h=\operatorname{Ad}(u_h)$ ($h \in \Gamma$), $A=Q$, $g=s$ and $w=u_s$; note that
$\theta_s(u_s)=u_s$.
\end{proof}

\begin{cor}\label{cor:cs-tensor}
Let $\sigma$ be a trace-preserving action of $\Gamma$ on a tracial von Neumann algebra $(Q,\tau_Q)$ and let
$\cP\subseteq L(\Gamma)\bar\otimes Q$ be a von Neumann subalgebra which is invariant under the diagonal
action and forms a commuting square with $\mathbb{C}1\bar{\otimes}Q$. Write
$\cP\cap(\mathbb{C}1\bar{\otimes}Q)=\mathbb{C}1\bar{\otimes} B$. If $\sigma$ is pointwise weakly mixing relative to $B$, then $\cP\subseteq L(\Gamma)\bar\otimes B$.
\end{cor}
\begin{proof}
The argument is the same as for Corollary~\ref{cor:cs-crossed}. Let
$x=\sum_{t\in\Gamma}\lambda_t\otimes f_t\in\cP$, with $f_t\in Q$. For $t\in\Gamma$ and $f\in Q$ with
$\mathbb{E}_B(f)=0$ we have
$\tau_Q(f^*f_t)=\langle x,\lambda_t\otimes f\rangle=\langle x,\mathbb{E}_{\cP}(\lambda_t\otimes f)\rangle$,
so it is enough to show that $\mathbb{E}_{\cP}(\lambda_t\otimes f)=0$. For $t=e$ this follows from
Lemma~\ref{lem:commuting-square}. For $t\neq e$, we apply Theorem~\ref{thm:cs-engine} with
$M=L(\Gamma)\bar\otimes Q$, the diagonal action $\theta$, $A=\mathbb{C}1\bar{\otimes}Q$, $g=t$ and
$w=\lambda_t\otimes1$; note that $\theta_t(\lambda_t\otimes1)=\lambda_{ttt^{-1}}\otimes1=\lambda_t\otimes1$
and $(1\otimes f)w=\lambda_t\otimes f$.
\end{proof}

Specialising Corollaries~\ref{cor:cs-crossed} and~\ref{cor:cs-tensor} to the case $B=\C1$ gives Theorem~\ref{thm:scalar-expectation} from the introduction.

%\begin{proposition}\label{prop:conditional-character-proof}
%Let $\Gamma$ be a countable discrete group and let $\sigma$ be a trace preserving action of $\Gamma$ on a tracial von Neumann algebra $(Q,\tau_Q)$ such that $\sigma_g$ is weakly mixing for every $g\neq e$. If
%$\cP\subseteq Q\rtimes_\sigma\Gamma$ is a $\Gamma$-invariant von Neumann subalgebra satisfying
%$\mathbb{E}(\cP)=\C1$, then $\cP\subseteq L(\Gamma)$. The same assertion holds for a subalgebra
%$\cP\subseteq L(\Gamma)\bar\otimes Q$ invariant under the diagonal action and satisfying
%$\mathbb{E}(\cP)=\C1$. In both cases, if $\Gamma$ has the ISR property, then $\cP=L(N)$ for some
%$N\trianglelefteq\Gamma$. 
%\end{proposition}
\begin{proof}[Proof of Theorem~\ref{thm:scalar-expectation}]
Since $\mathbb{E}(\cP)=\C1\subseteq\cP$, the algebra $\cP$ forms a commuting square with $Q$
(respectively, with $\mathbb{C}1\bar{\otimes}Q$), and the intersection is $\C1$. Weak mixing relative to $\C1$
is weak mixing, so Corollary~\ref{cor:cs-crossed} (respectively, Corollary~\ref{cor:cs-tensor})
gives $\cP\subseteq\C1\rtimes\Gamma=L(\Gamma)$ (respectively, $\cP\subseteq L(\Gamma)\bar\otimes\C1$).
In either case $\cP$ is a $\Gamma$-invariant von Neumann subalgebra of $L(\Gamma)$, and the last
assertion is the definition of the ISR property.
\end{proof}

\begin{comment}

\begin{remark}\label{rem:cs-improves}
Corollary~\ref{cor:cs-crossed} contains Proposition~\ref{prop:conditional-character-proof} as the
case $B=\C1$. When $\mathbb{E}(\cP)\subseteq\cP$, it also gives the conclusion of
Theorem~\ref{thm:gen} below without the ECI property and without assuming that $\sigma_g$ itself is
weakly mixing. No ICC assumption is needed in Proposition~\ref{prop:conditional-character-proof}.
Example~\ref{rem:nonicc-ideals} shows that, for non-ICC groups, an invariant subalgebra with scalar
intersection with $\Linf$ need not form a commuting square with $\Linf$.
\end{remark}

\end{comment}

\section{The crossed product case}\label{sec:crossed}

In this section we prove Theorem~\ref{thm:main}. We first record the intermediate-algebra argument for trace-preserving actions.

\begin{lem}\label{lem:groupoid-to-normal}
Let $\sigma$ be a trace-preserving action of a countable group $\Gamma$ on a tracial
von Neumann algebra $(B,\tau_B)$. Suppose that $B'\cap(B\rtimes\Gamma)=Z(B)$ and that
$\sigma_s$ is ergodic for every $s\neq e$. If $B\subseteq\cP\subseteq B\rtimes\Gamma$
is a $\Gamma$-invariant intermediate von Neumann algebra, then $\cP=B\rtimes N$ for the
normal subgroup $N=\{s\in\Gamma\mid u_s\in\cP\}$.
\end{lem}
\begin{proof}

Recall that the canonical conditional expectation $\mathbb{E}_{\cP}$ is $\Gamma$-equivariant, that is,
$\mathbb{E}_{\cP}(u_gxu_g^*)=u_g\mathbb{E}_{\cP}(x)u_g^*$ for all $g\in\Gamma$ and $x\in B\rtimes\Gamma$.

Fix $e\neq s\in \Gamma$ and let $v_s=u_s^*\mathbb{E}_{\cP}(u_s)$. We claim that
$v_s\in\cP'\cap(B\rtimes \Gamma)$. Indeed, for $m\in\cP$ let
$\tilde{m}=u_s^*mu_s\in\cP$. Using $\cP$-bimodularity of $\mathbb{E}_{\cP}$ twice, we obtain
\[
\mathbb{E}_{\cP}(u_s)\tilde{m}=\mathbb{E}_{\cP}(u_s\tilde{m})=\mathbb{E}_{\cP}(mu_s)=m\mathbb{E}_{\cP}(u_s),
\]
whence $v_s\tilde{m}=u_s^*m\mathbb{E}_{\cP}(u_s)=\tilde{m}v_s$. As $m\mapsto \tilde{m}$ is a bijection of $\cP$, the
claim follows. Since $B\subseteq\cP$, our assumption gives
$v_s\in\cP'\cap(B\rtimes\Gamma)\subseteq B'\cap(B\rtimes\Gamma)=Z(B)$.

%Next, $v_s$ is $\sigma_s$-invariant. 
By equivariance of $\mathbb{E}_{\cP}$ we have
$\mathbb{E}_{\cP}(u_s)=u_s\mathbb{E}_{\cP}(u_s)u_s^*$, that is,
$u_s^*\mathbb{E}_{\cP}(u_s)=\mathbb{E}_{\cP}(u_s)u_s^*$, and therefore $\sigma_s(v_s)=u_sv_su_s^*=\mathbb{E}_{\cP}(u_s)u_s^*=v_s$.
Since $\sigma_s$ is ergodic on $B$, we get $v_s=c_s1$ for some $c_s\in\C$, that is,
$\mathbb{E}_{\cP}(u_s)=c_su_s$. Applying $\mathbb{E}_{\cP}$ once more and using idempotence yields
$c_s^2=c_s$, so $c_s\in\{0,1\}$. Thus for every $s\in \Gamma$, either $u_s\in\cP$ or
$\mathbb{E}_{\cP}(u_s)=0$.

Let $N=\{s\in \Gamma\mid u_s\in\cP\}$. Then $N$ is a subgroup, and it is normal because
$u_gu_su_g^*=u_{gsg^{-1}}$ for all $s, g \in \Gamma$ and $\cP$ is $\Gamma$-invariant. Clearly
$B\rtimes N\subseteq\cP$. Conversely, let $x\in\cP$ and $s\in \Gamma \setminus  N$. For every $b\in B$,
bimodularity gives $\mathbb{E}_{\cP}(bu_s)=b\mathbb{E}_{\cP}(u_s)=0$, hence
$\langle x,bu_s\rangle=\langle x,\mathbb{E}_{\cP}(bu_s)\rangle=0$.
Thus the $s$-th Fourier coefficient of $x$ vanishes. Hence $x\in B\rtimes N$, and
$\cP=B\rtimes N$.
\end{proof}

\begin{thm}\label{thm:tracial-crossed-classification}
Let $\sigma$ be a trace-preserving action of a countable discrete group $\Gamma$ on
$(Q,\tau_Q)$, let $\cP\subseteq Q\rtimes\Gamma$ be a $\Gamma$-invariant von Neumann
subalgebra, and let $B=\cP\cap Q$. Suppose that $B\subseteq Q$ is centrally
$\Gamma$-free and that $\sigma$ is pointwise weakly mixing on $Q$ and pointwise
weakly mixing relative to $B$. Then $\cP=B\rtimes N$ for the normal subgroup
$N=\{s\in\Gamma\mid u_s\in\cP\}$.
\end{thm}
\begin{proof}
By Lemma~\ref{lem:central-free-expectation}, $\mathbb{E}(\cP)\subseteq\cP$ and
$B'\cap(B\rtimes\Gamma)=Z(B)$. Corollary~\ref{cor:cs-crossed} gives
$B\subseteq\cP\subseteq B\rtimes\Gamma$. Since each $\sigma_s$, $s\neq e$, is
weakly mixing on $Q$, its restriction to $B$ is ergodic. Apply
Lemma~\ref{lem:groupoid-to-normal}.
\end{proof}

For an essentially free p.m.p.~action on $(Y,\eta)$, the algebra
$B=L^\infty(Y,\eta)$ is maximal abelian in $B\rtimes\Gamma$, by
\cite[Section~1.4.5(i)]{AnantharamanDelarochePopaII1}. Thus, the relative commutant
assumption in Lemma~\ref{lem:groupoid-to-normal} holds. Under the hypotheses of
Theorem~\ref{thm:main}, the ergodicity assumption also holds, since weak mixing
passes to factors and implies ergodicity.

\begin{proof}[Proof of Theorem~\ref{thm:main}]
Weak mixing relative to $Q$ itself holds trivially, since
$\mathbb{E}_Q(\xi)=0$ forces $\xi=0$. Hence condition~(a) implies that
$\sigma$ is pointwise weakly mixing relative to every $\Gamma$-factor of
$(Q,\tau_Q)$, including $\C1$ and $Q$.
We prove the two assertions separately.

\noindent
\textit{Case (ii).} If $\mathbb{E}(\cP)=\C1$, then Theorem~\ref{thm:scalar-expectation}
gives $\cP\subseteq L(\Gamma)$ and, when $\Gamma$ has the ISR property, $\cP=L(N)$ for a normal
subgroup $N\trianglelefteq\Gamma$.

\noindent
\textit{Case (i).} Let $B=\cP\cap Q$ and suppose that $B\neq\C1$. Since
$\cP$ and $Q$ are $\Gamma$-invariant, $B$ is a non-trivial $\Gamma$-factor of
$(Q,\tau_Q)$, so the inclusion $B\subseteq Q$ is centrally $\Gamma$-free by
condition~(b). By the first paragraph, $\sigma$ is pointwise weakly mixing on
$Q$ and pointwise weakly mixing relative to $B$.
Theorem~\ref{thm:tracial-crossed-classification} therefore yields
$\cP=B\rtimes N$ for the normal subgroup $N=\{s\in\Gamma\mid u_s\in\cP\}$.
\end{proof}

\begin{remark}\label{rem:open-case-cs}
Assume condition~(b) in Theorem~\ref{thm:main}. By
Lemma~\ref{lem:central-free-expectation} and Remark~\ref{rem:cs-crossed}, $\cP$ forms a commuting square with $Q$ whenever
$\cP\cap Q\neq\C1$, and the same holds trivially when $\mathbb{E}(\cP)=\C1$. Conversely, if $\cP$
forms a commuting square with $Q$, then $\mathbb{E}(\cP)=\cP\cap Q$ by
Remark~\ref{rem:cs-crossed}, so either $\cP\cap Q\neq\C1$ or $\mathbb{E}(\cP)=\C1$. Hence the case outside the subalgebra alternatives in parts~(i) and~(ii) of
Theorem~\ref{thm:main} is exactly the case in which $\cP$ and $Q$ do not form a commuting square. 
\end{remark}

We shall now discuss some cases where the conditions above do not hold. We begin with the non-ICC case. The ideal construction in
\cite[Theorem~5.12]{amrutam2024nonabelian} motivates the following example.

\begin{example}\label{rem:nonicc-ideals}\upshape
Let $\Gamma$ be a non-trivial countable group that is not ICC, and let
$\Gamma\curvearrowright(X,\nu)$ be a p.m.p.~action on a nonatomic
probability space. Then $\Linf\rtimes\Gamma$ admits a $\Gamma$-invariant von Neumann subalgebra that contains
$L(\Gamma)$ and is not of the form $B\rtimes\Lambda$ in the canonical embedding.

Indeed, let $p \in Z(L(\Gamma))$ be a non-trivial projection. 
%the sum of the group unitaries over a finite nonidentity
%conjugacy class is a nonscalar element of $Z(L(\Gamma))$.
%Taking a nonscalar self-adjoint part and a spectral projection gives
%$p\in Z(L(\Gamma))$ with $0<p<1$. 
Set $\cP=p(\Linf\rtimes\Gamma)p\oplus p^{\perp}L(\Gamma)$. Then % two summands have orthogonal units
%$p$ and $q$, so 
$\cP$ is a $\Gamma$-invariant von Neumann subalgebra which contains $L(\Gamma)$.% and is , since every group
%unitary commutes with $p$ and $q$.

We claim that $\cP\cap \Linf=\C1$. If $a\in\cP\cap \Linf$, then $a$
commutes with $p^{\perp}$ and $p^{\perp}ap^{\perp}$ belongs to $p^{\perp}L(\Gamma)$. Bimodularity of the
trace-preserving expectation onto $L(\Gamma)$ gives
$p^{\perp}ap^{\perp}=\mathbb{E}_{L(\Gamma)}(p^{\perp}ap^{\perp})=\nu(a)p^{\perp}$.
Thus, $(a-\nu(a)1)p^{\perp}=0$. Applying $\mathbb{E}$ to
$(a-\nu(a)1)p^{\perp}(a-\nu(a)1)^*$ gives
$\tau(p^{\perp})|a-\nu(a)1|^2=0$, proving the claim.

On the other hand, $\mathbb{E}(\cP)\ne\C1$. Indeed, write
$p=\sum_{g\in\Gamma}p_g u_g$ and choose a measurable set $D$ with
$0<\nu(D)<\tau(p)$. Then
\[
 \mathbb{E}(p1_Dp)=\sum_{g\in\Gamma}|p_g|^2\sigma_g(1_D)
 \geq |p_e|^2\sigma_e(1_D)=\tau(p)^2 1_D.
\]
To justify the first identity, let $p_F$ be the finite Fourier sums of $p$.
They converge to $p$ in $L^2(\Linf\rtimes\Gamma)$, so $p_F1_Dp_F^*$ converges to
$p1_Dp$ in $L^1(\Linf\rtimes\Gamma)$ by the Cauchy--Schwarz inequality. Apply the $L^1$-contractive
expectation and note that the series on the right converges in
operator norm, since $\sum_{g\in \Gamma}|p_g|^2=\tau(p)$.
If this expectation were scalar, trace preservation would imply that it equals
$\tau(p)\nu(D)1$, contrary to the lower bound on $D$.
Finally, an equality $\cP=B\rtimes\Lambda$ in the canonical
embedding would force $\Lambda=\Gamma$, since $L(\Gamma)\subseteq\cP$,
and $B=\C1$, since $\cP\cap \Linf=\C1$. This contradicts
$\mathbb{E}(\cP)\ne\C1$.
\begin{comment}
Note already here The tensor product has the same obstruction. Let $\mathcal Q$ be the direct sum
of $pL(\Gamma)\bar\otimes\Linf$ and $qL(\Gamma)\bar\otimes\C1$.
It is invariant under the diagonal action and contains
$L(\Gamma)\bar{\otimes}\mathbb{C}1$. Its intersection with $\mathbb{C}1\bar{\otimes} \Linf$ is scalar. If $1\otimes f\in\mathcal Q$, then
$q\otimes f\in qL(\Gamma)\bar\otimes\C1$, and applying the normalized
trace on $qL(\Gamma)$ to the first tensor factor shows that $f$ is
scalar. Moreover $\mathbb{E}(p\otimes f)=\tau(p)f$ for every $f\in \Linf$, so
$\mathbb{E}(\mathcal Q)=\Linf$. As above, these facts rule out an equality
$\mathcal Q=L(\Lambda)\bar\otimes B$ in the canonical embedding.
\end{comment}
\end{example}

The following example shows that the relative weak mixing hypothesis in
Theorem~\ref{thm:main} cannot be removed.
\begin{example}
\label{ex:relative-mixing-needed}
Let
$\Gamma=\F_2$ be the free group on two generators, denoted $a$ and $b$.
Let $X=\{\pm1\}^{\Gamma}$ be equipped with the standard symmetric product measure $\nu$ and let $\Gamma$ act on $(X,\nu)$ by shifts.
 Write
$r_h(x)=x_h$ for $h\in\Gamma$, $x \in X$, so that $r_h \in \Linf$.
Then $\sigma_s(r_h)=r_{sh}$ for $s,h\in\Gamma$. For $h\in\Gamma$, let
$c_h=hah^{-1}$ and $q_h=r_hu_{c_h}\in\Linf\rtimes\Gamma$.
Let $\cP=W^*(q_h\mid h\in\Gamma)\subseteq\Linf\rtimes\Gamma$. Then $\cP$ is $\Gamma$-invariant. Indeed, for $s,h \in \Gamma$ we have
\[
u_s q_h u_s^*
=
u_s r_h u_{c_h}u_s^*
=
\sigma_s(r_h)u_{s c_hs^{-1}}
=
r_{sh}u_{(sh)a(sh)^{-1}}
=
q_{sh}.
\]
Let
$B=\cP\cap \Linf$. Then $B\neq\mathbb C1$. Indeed, for every $h\in\Gamma$, $c_h=c_{ha}$, so
$r_hr_{ha}=q_hq_{ha}^*\in B$.
This function is nonscalar.
We now show that $\cP$ is not of the form $B\rtimes N$. Suppose, to the contrary, that
$\cP=B\rtimes N$
for some normal subgroup $N\trianglelefteq\Gamma$. Since
$q_e=r_eu_a\in\cP$
has Fourier support at $a$, we must have $a\in N$. Hence
$u_a\in B\rtimes N=\cP$.
Therefore $r_e=q_eu_a^*\in\cP\cap\Linf=B$.
We shall see that this is impossible, using the identification
$\Linf\rtimes\Gamma\cong L((\bigoplus_{\Gamma}\mathbb Z/2\mathbb Z)\rtimes\Gamma)$.
Under this identification, for every $h\in\Gamma$, $r_hu_{c_h}$ is the canonical unitary associated with the group element
$(\delta_h,c_h)$, where $\delta_h\in\bigoplus_\Gamma\mathbb Z/2\mathbb Z$ has value $1$ at $h$ and $0$ elsewhere.
Let $K=\langle(\delta_h,c_h)\mid h\in\Gamma\rangle$, a subgroup of
$(\bigoplus_{\Gamma}\mathbb Z/2\mathbb Z)\rtimes\Gamma$. Then
$\cP=L(K)$,
and, since $B=\cP\cap\Linf$, this identification gives $B=L\left(K\cap\bigoplus_{\Gamma}\mathbb Z/2\mathbb Z\right)$. Define a homomorphism $\chi$ from
$(\bigoplus_{\Gamma}\mathbb Z/2\mathbb Z)\rtimes\Gamma$ to $\mathbb Z/2\mathbb Z$
by $\chi(\eta,g)=|\eta|+\operatorname{exp}_a(g)\pmod2$ for
$\eta\in\bigoplus_{\Gamma}\mathbb Z/2\mathbb Z$ and $g\in\Gamma$. Here $|\eta|$ denotes the number of nonzero coordinates of $\eta$, and $\operatorname{exp}_a(g)$ is the $a$-exponent sum of $g$, namely the image of $g$ under the surjective homomorphism $\Gamma\twoheadrightarrow\mathbb{Z}$ sending $a$ to $1$ and $b$ to $0$. For every generator $(\delta_h,c_h)$ of $K$, we have
$|\delta_h|=1$, and
$\operatorname{exp}_a(c_h)=1$,
and hence $\chi(\delta_h,c_h)=0$.
Thus,
$K\subseteq\ker\chi$. But
$\chi(\delta_e,e)=1$, so 
$(\delta_e,e)\notin K$.
Equivalently,
$r_e\notin B$.
This contradicts the conclusion above. Hence $\cP\neq B\rtimes N$ for every normal subgroup $N\trianglelefteq\Gamma$.

Finally, we show explicitly that relative weak mixing over $B$ fails. Since
$c_e=c_{a^n}=a$
for every $n\ge1$, we have $q_e q_{a^n}^*
=
r_e r_{a^n}\in B$.
Moreover, from the description of $B$ above,
$\mathbb{E}_B(r_e)=0$. But
$\sigma_{a^n}(r_e)=r_{a^n}$. Therefore,
$\mathbb{E}_B\bigl(r_e^*\sigma_{a^n}(r_e)\bigr) = \mathbb{E}_B(r_e r_{a^n}) = r_e r_{a^n}.$
These functions have $L^2$ norm $1$ for every $n\ge1$, so $\sigma_a$ is not weakly mixing relative to $B$.
\end{example}

\begin{remark}\label{rem:bernoulli-scope}
Suppose that a compact second countable group $K$ acts non-trivially on $(X,\nu)$ by transformations $\kappa_k$, $k\in K$, which preserve $\nu$.
Assume that this action commutes with the p.m.p.\ action of $\Gamma$ and that its Koopman representation is strongly continuous. Then
$B=L^\infty(X,\nu)^K$ is a proper $\Gamma$-factor. If $\sigma_g$ is ergodic for a fixed $g\in\Gamma$, then $\sigma_g$ cannot be weakly mixing relative to $B$.

To see this, let $\pi$ denote the factor map from $X$ to $Y$ representing $B$, and let
$\lambda=\nu\times_Y\nu$ (see Remark \ref{remark: relative weak mixing for diagonal action wrt a coordinate factor} for the notation used). For $k\in K$, let
$\lambda_k=(\mathrm{id},\kappa_k)_*\nu$. Haar averaging over $K$ yields the conditional
expectation onto $B$, and hence, for $f,h\in L^\infty(X,\nu)$,
\[
 \int_K\int_{X\times X} f(x)h(x')d\lambda_k(x,x')dk
 =\int_X f \mathbb{E}_B(h)d\nu
 =\int_{X\times X}f(x)h(x')d\lambda(x,x').
\]
Thus, $\lambda=\int_K\lambda_kdk$. Every $\lambda_k$ is invariant under
$\sigma_g\times\sigma_g$ ($k\in K$, $g \in \Gamma$), since $\kappa_k$ commutes with $\sigma_g$.
If $\sigma_g$ were weakly mixing relative to $B$, then $\lambda$ would be
an ergodic invariant probability measure, by
Remark~\ref{remark: relative weak mixing for diagonal action wrt a coordinate factor}.
Ergodic invariant probability measures are extreme among invariant
probability measures, so this decomposition gives $\lambda_k=\lambda$
for Haar almost every $k\in K$. For each $f\in L^{\infty}(X,\nu)$, the function
$k\mapsto\int_X f(x)\overline{f(\kappa_kx)}d\nu(x)$ is continuous by strong continuity of the Koopman representation. Its
almost everywhere constant value must therefore also be its value at
the identity. It follows that
\[
 \|f\|_2^2
 =\int_{X\times X}f(x)\overline{f(x')}d\lambda(x,x')
 =\|\mathbb{E}_B(f)\|_2^2.
\]
As $\mathbb{E}_B$ acts as an orthogonal projection on $L^2(X,\nu)$, this forces every
bounded $f$ to belong to $B$, contradicting properness.

In particular, this obstruction applies to Bernoulli shifts whose base admits a
non-trivial compact group of transformations preserving the base measure, including
the symmetric two-point base and, more generally, a compact group equipped with Haar measure.
%Essentially free, prime, totally weakly mixing actions satisfy the hypotheses of
%Theorem~\ref{thm:main}, since their only proper factor is trivial.
\end{remark}

\subsection{Beyond the commuting square}\label{subsec:beyond}
When $\cP$ and $Q$ do not form a commuting square, the ECI property allows us to
control the Fourier coefficients of elements of $\cP$ by the von Neumann algebra
generated by $\mathbb{E}(\cP)$.

\begin{thm}\label{thm:gen}
Let $\Gamma$ be a countable discrete group with the ECI property, and let $\sigma$
be a trace-preserving action of $\Gamma$ on $(Q,\tau_Q)$. Let
$\cP\subseteq Q\rtimes\Gamma$ be a $\Gamma$-invariant von Neumann subalgebra, and let
$B=\mathbb{E}(\cP)''$. If $\sigma$ is pointwise weakly mixing relative to $B$, then
$\cP\subseteq B\rtimes\Gamma$.
\end{thm}
\begin{proof}
Equivariance of $\mathbb{E}$ implies that $B$ is $\Gamma$-invariant. Fix $x\in\cP$
with $\|x\|_\infty\le1$ and write $f_g=\mathbb{E}(xu_g^*)$ for $g\in\Gamma$.
We have $f_e\in B$. Fix $t_0\neq e$ and $\epsilon>0$. Choose a finite set
$F\subseteq\Gamma$ containing $e$ and $t_0$ such that $\|x-x_F\|_2<\epsilon$,
where $x_F=\sum_{g\in F}f_gu_g$, and let $F_1=F\cap C_\Gamma(t_0)$.
By ECI, there is $n_0\in\N$ such that
$t_0^n(F\setminus F_1)t_0^{-n}\cap F=\varnothing$ for every $n\ge n_0$.
For such $n$, let $c_n=\mathbb{E}(u_{t_0}^nxu_{t_0}^{-n}x^*)$.
Since $\cP$ is invariant, $c_n\in B$ and $\|c_n\|_\infty\le1$.
Multiplying the finite Fourier sums gives
\[
 \mathbb{E}(u_{t_0}^nx_Fu_{t_0}^{-n}x_F^*)
 =\sum_{g\in F_1}\sigma_{t_0}^n(f_g)f_g^*.
\]
By equation~\eqref{eq:basic}, we have
\begin{equation}\label{eq:L1A}
 \left\|c_n-\sum_{g\in F_1}\sigma_{t_0}^n(f_g)f_g^*\right\|_1
 \le\|x-x_F\|_2\|x\|_2+\|x_F\|_2\|x-x_F\|_2<2\epsilon.
\end{equation}

Let $\xi\in Q$ satisfy $\mathbb{E}_B(\xi)=0$ and $\|\xi\|_\infty\le1$.
Apply the functional $a\mapsto\tau_Q(\xi^*\sigma_{t_0}^n(\xi)a)$, whose norm on
$L^1(Q,\tau_Q)$ is at most one, to equation~\eqref{eq:L1A}.
Since $c_n\in B$ and $\|c_n\|_2\le1$, we have
\[
 |\tau_Q(\xi^*\sigma_{t_0}^n(\xi)c_n)|
 =|\tau_Q(c_n\mathbb{E}_B(\xi^*\sigma_{t_0}^n(\xi)))|
 \le\|\mathbb{E}_B(\xi^*\sigma_{t_0}^n(\xi))\|_2.
\]
Relative weak mixing and the Cauchy--Schwarz inequality show that the Ces\'aro averages of these
absolute values tend to zero. For each term in the finite sum, traciality gives
\[
 \tau_Q(\xi^*\sigma_{t_0}^n(\xi)\sigma_{t_0}^n(f_g)f_g^*)
 =\tau_Q((\xi f_g)^*\sigma_{t_0}^n(\xi f_g)).
\]
By the mean ergodic theorem, the Ces\'aro averages on the right converge to the
squared norm of the orthogonal projection of $\xi f_g$ onto the invariant vectors
in $L^2(Q,\tau_Q)$. This limit is nonnegative and at least
$|\tau_Q(\xi f_g)|^2$, since the invariant subspace contains the constants.
Averaging the scalar inequality obtained from equation~\eqref{eq:L1A} and passing
to the limit therefore gives
$\sum_{g\in F_1}|\tau_Q(\xi f_g)|^2\le2\epsilon$.
Since $t_0\in F_1$ and $\epsilon$ was arbitrary, $\tau_Q(\xi f_{t_0})=0$.
By scaling, this holds for every $\xi\in Q$ with $\mathbb{E}_B(\xi)=0$.
Taking $\xi=(f_{t_0}-\mathbb{E}_B(f_{t_0}))^*$ gives
$\|f_{t_0}-\mathbb{E}_B(f_{t_0})\|_2^2=0$, so $f_{t_0}\in B$.
As $t_0\neq e$ was arbitrary, every Fourier coefficient of $x$ belongs to $B$,
and hence $x\in B\rtimes\Gamma$.
\end{proof}

\begin{remark}\label{rem:beyond}
If $\mathbb{E}(\cP)\subseteq\cP$, then $\mathbb{E}(\cP)''=\mathbb{E}(\cP)$, and
Corollary~\ref{cor:cs-crossed} gives the conclusion of Theorem~\ref{thm:gen}
without the ECI property. If $\mathbb{E}(\cP)''=Q$, the conclusion is trivial.
In particular, for a prime p.m.p.~action, $\mathbb{E}(\cP)''=\Linf$ as soon as
$\mathbb{E}(\cP)\neq\C1$.
\end{remark}
\section{The tensor product case}\label{sec:tensor}
We now turn our attention to the tensor product case. 
While the results of Section~\ref{sec:commuting-square} on $L(\Gamma)\bar\otimes Q$
and Proposition~\ref{prop:insideamenradtensor} below hold for an
arbitrary tracial von Neumann algebra $(Q,\tau_Q)$, we restrict ourselves here to $Q=\Linf$ for the following reason. By
Corollary~\ref{cor:cs-tensor}, a key point is to show that an invariant
subalgebra $\cP$ forms a commuting square with $\mathbb{C}1\bar\otimes Q$; in
the crossed product, this is exactly where central freeness was used
(Lemma~\ref{lem:central-free-expectation}). When $Q=\Linf$ is commutative, the
algebra $\mathbb{C}1\bar\otimes\Linf$ is contained in the centre of
$L(\Gamma)\bar\otimes\Linf$, so $\mathbb{E}_{\cP}$ maps it into the centre
$Z(\cP)$. Consequently, the commuting square is automatic as soon as
$Z(\cP)\subseteq\mathbb{C}1\bar\otimes\Linf$
(Lemma~\ref{lem:tensor-center-valued-character} and
Remark~\ref{rem:tensor-cs}), and this inclusion holds whenever $\Gamma$ has
trivial amenable radical (Corollary~\ref{cor:diagonal-tensor-center}). For
noncommutative $Q$, the algebra $\mathbb{C}1\bar\otimes Q$ is not central in
$L(\Gamma)\bar\otimes Q$, and this mechanism is not available.

Throughout this section, $\sigma$ is a p.m.p.~action on $(X,\nu)$,
and $L(\Gamma)\bar\otimes\Linf$ carries the diagonal action and trace
from Section~\ref{sec:prelim}. We write
$x=\sum_{t\in\Gamma}\lambda_t\otimes f_t$ for its $L^2$ Fourier
expansion.

The first lemma determines the expectation of functions in $\Linf$  onto a von Neumann subalgebra
$\cP$ with a specific position of the centre (invariance is not needed for this assertion).

\begin{lem}\label{lem:tensor-center-valued-character}
Let $\cP\subseteq L(\Gamma)\bar\otimes\Linf$ be a von Neumann
subalgebra such that $Z(\cP)\subseteq\mathbb{C}1\bar{\otimes}\Linf$, and write
$Z(\cP)=\mathbb{C}1\bar{\otimes} B$ for a von Neumann subalgebra $B\subseteq\Linf$.
Then $\mathbb{E}_{\cP}(1\otimes f)=1\otimes\mathbb{E}_B(f)$ for
all $f\in\Linf$. In particular, if $\cP$ is a factor, then
$\mathbb{E}_{\cP}(1\otimes f)=\nu(f)1$. If $\cP$ is
$\Gamma$-invariant, then $B$ is $\Gamma$-invariant.
\end{lem}
\begin{proof}
For $f\in\Linf$ and $x\in\cP$, centrality of $1\otimes f$ and bimodularity of $\mathbb{E}_{\cP}$ give
\[
 x\mathbb{E}_{\cP}(1\otimes f)
 =\mathbb{E}_{\cP}(x(1\otimes f))
 =\mathbb{E}_{\cP}((1\otimes f)x)
 =\mathbb{E}_{\cP}(1\otimes f)x.
\]
Thus, $\mathbb{E}_{\cP}(1\otimes f)=1\otimes b_f$ for some $b_f\in B$.
For every $b\in B$, trace preservation gives
$\nu(b^*b_f) =\tau\bigl((1\otimes b)^*\mathbb{E}_{\cP}(1\otimes f)\bigr) =\nu(b^*f)$.
This characterizes $b_f=\mathbb{E}_B(f)$. The factor assertion follows
by taking $B=\C1$. If $\cP$ is invariant, then $\theta_s$ preserves
$Z(\cP)=\mathbb{C}1\bar{\otimes} B$, so $\sigma_s(B)=B$ for every $s\in\Gamma$.
\end{proof}

\begin{remark}\label{rem:tensor-cs}
Lemma~\ref{lem:tensor-center-valued-character} shows that under the stated assumptions
$\mathbb{E}_{\cP}$ preserves $\mathbb{C}1\bar{\otimes}\Linf$. By
Lemma~\ref{lem:commuting-square}, these algebras form a commuting square,
and
$$\cP\cap(\mathbb{C}1\bar{\otimes}\Linf) =\mathbb{E}_{\cP}(\mathbb{C}1\bar{\otimes}\Linf)=\mathbb{C}1\bar{\otimes} B=Z(\cP).$$
\end{remark}

\begin{lem}\label{lem:tensor-group-expectation}
Let $\Gamma$ be a countable discrete group and let
$\Gamma\curvearrowright(X,\nu)$ be a p.m.p.~action, denoted by $\sigma$, which is
pointwise weakly mixing.
If $\cP\subseteq L(\Gamma)\bar\otimes\Linf$ is a von Neumann subalgebra invariant under the
diagonal action, then
$\mathbb{E}_{\cP}(L(\Gamma)\bar{\otimes}\mathbb{C}1)\subseteq L(\Gamma)\bar{\otimes}\mathbb{C}1$ and $\mathbb{E}_{L(\Gamma)}(\cP)\subseteq\cP$.
\end{lem}
\begin{proof}
Fix $g\neq e$ and write
$x=\mathbb{E}_{\cP}(\lambda_g\otimes1)
  =\sum_{t\in\Gamma}\lambda_t\otimes f_t$.
Equivariance gives $\theta_g(x)=x$, and comparison of Fourier
coefficients yields $f_{gtg^{-1}}=\sigma_g(f_t)$ for every $t\in\Gamma$.
On an infinite orbit under conjugation by $g$, all coefficients have
the same $L^2$ norm. Since $\sum_t\|f_t\|_2^2=\|x\|_2^2<\infty$,
these coefficients vanish. On an orbit of finite length $m$,
$\sigma_g^m(f_t)=f_t$. Every positive power of a weakly mixing
transformation is ergodic, so $f_t$ is scalar. Indeed, a nonzero
mean-zero function fixed by $\sigma_g^m$ would generate a nonzero
finite-dimensional invariant subspace under $\sigma_g$, contradicting
weak mixing. Thus, every Fourier coefficient of $x$ is scalar, and
$x\in L^2(L(\Gamma)\bar{\otimes}\mathbb{C}1)$. As $x$ is bounded, it belongs to
$L(\Gamma)\bar{\otimes}\mathbb{C}1$. The case $g=e$ follows from
$\mathbb{E}_{\cP}(1)=1$.

%Finite linear combinations of the group unitaries are dense in
%$L^2(L(\Gamma)\bar{\otimes}\mathbb{C}1)$. 
The $L^2$-continuity of
$\mathbb{E}_{\cP}$ shows that it preserves $L^2(L(\Gamma)\bar{\otimes}\mathbb{C}1)$. For a bounded element of $L(\Gamma)\bar{\otimes}\mathbb{C}1$, its expectation
is bounded as well, so it belongs to $L(\Gamma)\bar{\otimes}\mathbb{C}1$. 
This proves
the first containment. Lemma~\ref{lem:commuting-square}, applied to
$\cP$ and $L(\Gamma)\bar{\otimes}\mathbb{C}1$, now gives the second containment.
\end{proof}

\subsection{Classification under the centre-position assumption}

The centre of $\cP$ determines the function-algebra part of the desired tensor decomposition.
If $Z(\cP)=\mathbb{C}1\bar{\otimes} B$, Lemma~\ref{lem:tensor-center-valued-character}
gives $\mathbb{E}_{\cP}(\mathbb{C}1\bar{\otimes}\Linf)=\mathbb{C}1\bar{\otimes} B$.
The resulting commuting square allows relative weak mixing to place
every Fourier coefficient in $B$.

\begin{thm}\label{thm:tensor-center-in-A}
Let $\Gamma$ be a countable discrete group with the ISR property. Let
$\Gamma\curvearrowright(X,\nu)$ be a p.m.p.~action which is pointwise weakly mixing relative to every proper $\Gamma$-factor of $\Linf$, including the
trivial factor. Let $\cP\subseteq L(\Gamma)\bar\otimes \Linf$ be a $\Gamma$-invariant von Neumann
subalgebra such that $Z(\cP)\subseteq\mathbb{C}1\bar{\otimes}\Linf$, i.e.~$Z(\cP)=\mathbb{C}1\bar{\otimes} B$ for a
$\Gamma$-invariant von Neumann subalgebra $B\subseteq\Linf$. Then there exists a normal subgroup
$N\trianglelefteq\Gamma$ such that $\cP=L(N)\bar\otimes B$.
\end{thm}
\begin{proof}
If $\Linf=\C1$, the assertion follows directly from the ISR property. Otherwise the
trivial factor is proper, so $\sigma$ is pointwise weakly mixing.
Set $\cR=\mathbb{E}_{L(\Gamma)}(\cP)$. By
Lemma~\ref{lem:tensor-group-expectation}, $\cR\subseteq\cP$.
Since $\cR\subseteq L(\Gamma)\bar{\otimes}\mathbb{C}1$,
we have $\cR\subseteq\cP\cap(L(\Gamma)\bar{\otimes}\mathbb{C}1)$. Conversely, if $x\in\cP\cap(L(\Gamma)\bar{\otimes}\mathbb{C}1)$, then
$\mathbb{E}_{L(\Gamma)}(x)=x$, and hence $x\in \cR$. Thus, $\cR=\cP\cap(L(\Gamma)\bar{\otimes}\mathbb{C}1)$. In particular,
$\cR$ is a $\Gamma$-invariant von Neumann subalgebra of $L(\Gamma)$, and by the ISR property $\cR=L(N)$
for some normal subgroup $N\trianglelefteq\Gamma$.

We claim that $\cP\subseteq L(\Gamma)\bar\otimes B$. By Remark~\ref{rem:tensor-cs}, $\cP$ forms a
commuting square with $\mathbb{C}1\bar{\otimes}\Linf$ and $\cP\cap(\mathbb{C}1\bar{\otimes}\Linf)=\mathbb{C}1\bar{\otimes} B$. If $B=\Linf$ there is
nothing to prove. Otherwise $B$ is a proper $\Gamma$-factor, $\sigma$ is pointwise weakly mixing relative
to $B$ and the claim follows from Corollary~\ref{cor:cs-tensor}.

Next, we claim that $\cP\subseteq L(N)\bar\otimes B$. Let $t\notin N$. First note that
$\mathbb{E}_{\cP}(\lambda_t\otimes1)=0$. Indeed, for every $z\in\cP$ we have
$\langle\lambda_t\otimes1,z\rangle=\langle\lambda_t\otimes1,\mathbb{E}_{L(\Gamma)}(z)\rangle=0$, since
$\mathbb{E}_{L(\Gamma)}(z)\in \cR=L(N)$ and $\lambda_t\perp L(N)$. Consequently, for every $b\in B$,
bimodularity gives
$\mathbb{E}_{\cP}(\lambda_t\otimes b)=\mathbb{E}_{\cP}(\lambda_t\otimes1)(1\otimes b)=0$, because
$1\otimes b\in Z(\cP)\subseteq\cP$. Now let $x=\sum_{s\in\Gamma}\lambda_s\otimes f_s\in\cP$. By the
previous claim $f_t\in B$, and therefore
$\|f_t\|_2^2=\langle x,\lambda_t\otimes f_t\rangle=\langle x,\mathbb{E}_{\cP}(\lambda_t\otimes f_t)\rangle=0$.
Thus, every Fourier coefficient of $x$ outside $N$ vanishes, which proves the claim.

Finally, $L(N)\bar{\otimes}\mathbb{C}1=\cR\subseteq\cP$ and $\mathbb{C}1\bar{\otimes} B=Z(\cP)\subseteq\cP$, so
$L(N)\bar\otimes B\subseteq\cP$. Together with the previous claim, we see that $\cP=L(N)\bar\otimes B$.
\end{proof}

We next remove the centre-position assumption when the \emph{amenable radical} of $\Gamma$, namely its largest amenable normal subgroup,
is trivial. We denote it by $\operatorname{Rad}(\Gamma)$. 
\subsection{General principle}\label{sec:center-position}

We adapt the Furstenberg-boundary hypertrace argument of Amrutam and Jiang
in the appendix to \cite{zhou2024noncommutative}, specifically Lemma~8.1
and the second proof of Theorem~F of that paper. 

The following result is a general statement concerning actions of $\Gamma$ on tracial von Neumann algebras. 

\begin{proposition}
\label{prop:insideamenradtensor}
Let $\Gamma$ be a countable discrete group, let $(Q,\tau_Q)$ be a von Neumann algebra with a faithful normal tracial state and let
$\sigma$ be a trace-preserving action of $\Gamma$ on $(Q,\tau_Q)$. Let
$M=L(\Gamma)\bar\otimes Q$ and let $\theta$ be again the diagonal action of $\Gamma$ on $M$, given by
$\theta_s(\lambda_t\otimes a)=\lambda_{sts^{-1}}\otimes\sigma_s(a)$
for $s,t\in\Gamma$ and $a\in Q$.
If $\cP\subseteq M$ is an amenable von Neumann subalgebra such that
$\theta_s(\cP)=\cP$ for every $s\in\Gamma$, then
$\cP\subseteq L(\operatorname{Rad}(\Gamma))\bar\otimes Q$.
\end{proposition}
\begin{proof}
By our standing assumption, $Q$ has separable predual.
Let $\widetilde{M}=M\rtimes_\theta\Gamma$, and denote by $v_s$, $s\in\Gamma$, the canonical
unitaries implementing the action inside $\widetilde{M}$, so that  $v_sxv_s^*=\theta_s(x)$ for every $x\in M$, $s \in \Gamma$. Hence when we view $\cP$ inside $\widetilde{M}$, the assumption
$\theta_s(\cP)=\cP$ becomes $v_s\cP v_s^*=\cP$. Write $u_t=\lambda_t\otimes1\in M$, $t \in \Gamma$. Then
$v_su_tv_s^*=u_{sts^{-1}}$.

The automorphism action is now represented as conjugation, while the Fourier support we want to control
is still the support in the $u_t$-coordinate. Apply the construction of  Subsection~\ref{subsec:boundary-rep}
with $Z=\partial_F\Gamma$. By minimality, every $x_0\in\partial_F\Gamma$
has a dense orbit, so the multiplication operators $m_f$ defined there give a faithful
representation of $C(\partial_F\Gamma)$ on $L^2(\widetilde M)$.
Lemma~\ref{lem:boundary-covariance} gives
$u_tm_fu_t^*=m_{t\cdot f}$ and $v_sm_fv_s^*=m_{s\cdot f}$ for all $s,t\in\Gamma$, $f \in C(\partial_F\Gamma)$.

Let $\widetilde\tau$ be the canonical trace on $\widetilde M$. Since $Q$
has separable predual and $\Gamma$ is countable, $\widetilde M$ also has
separable predual. Let
$\mathcal H=\operatorname{Hype}_{\widetilde\tau}(\cP)$ be the set of states
$\varphi$ on $\mathbb B(L^2(\widetilde M))$ satisfying
$\varphi|_{\widetilde M}=\widetilde\tau$ and $\varphi(aT)=\varphi(Ta)$ for all $a\in\cP$ and $T\in\mathbb B(L^2(\widetilde M))$.
By \cite[Proposition~2.4]{amrutam2025amenable}, $\mathcal H$ is nonempty.
These are weak$^*$-closed affine conditions
on the state space, so $\mathcal H$ is compact and convex. It is invariant
under the action given by $(s\cdot\varphi)(T)=\varphi(v_s^*Tv_s)$ ($s\in\Gamma$, $\varphi\in\mathcal H$, $T\in\mathbb B(L^2(\widetilde M))$) because $v_s$ normalises $\cP$ and $\widetilde\tau$ is a trace.
Restriction to the represented copy of $C(\partial_F\Gamma)$ is a
continuous affine map, and Lemma~\ref{lem:boundary-covariance} makes it
$\Gamma$-equivariant. Its image
$\mathcal K=\{\varphi|_{C(\partial_F\Gamma)}\mid\varphi\in\mathcal H\}\subseteq\operatorname{Prob}(\partial_F\Gamma)$
is therefore nonempty, compact, convex and $\Gamma$-invariant. Hence it equals
$\operatorname{Prob}(\partial_F\Gamma)$ (see for example \cite[Section~2]{amrutam2025amenable}).

\begin{comment}
We spell out the boundary argument underlying
\cite[Appendix, Lemma~8.1]{zhou2024noncommutative}. Choose
$\mu\in\mathcal K$. Strong proximality gives a point mass $\delta_y$ in
the weak$^*$ closure of $\Gamma\mu$, hence in $\mathcal K$.
Minimality and invariance then give $\delta_z\in\mathcal K$ for every
$z\in\partial_F\Gamma$. Finite convex combinations of point masses are
weak$^*$ dense in the probability measures on a compact Hausdorff space.
Since $\mathcal K$ is closed and convex, it follows that
$\mathcal K=\operatorname{Prob}(\partial_F\Gamma)$.
\end{comment}

Fix $t_0\in\Gamma\setminus\operatorname{Rad}(\Gamma)$. Since
$\operatorname{Rad}(\Gamma)=\ker(\Gamma\curvearrowright\partial_F\Gamma)$, we may choose $x\in \partial_F\Gamma$ such
that $t_0x\neq x$. Choose $f\in C(\partial_F\Gamma)$ with $0\leq f\leq1$, $f(x)=1$, and $f(t_0x)=0$. By the
previous paragraph we can select
$\varphi\in\operatorname{Hype}_{\widetilde\tau}(\cP)$ such that $\varphi|_{C(\partial_F\Gamma)}=\delta_x$.
Then $C(\partial_F\Gamma)$ is contained in the multiplicative domain of $\varphi$.
Let $a\in\cP$ and $b\in Q$. Write $f'=t_0^{-1}\cdot f$, so that $u_{t_0}^*(1\otimes b)=(1\otimes b)u_{t_0}^*$,
$u_{t_0}^*m_f=m_{f'}u_{t_0}^*$ and $u_{t_0}m_{f'}u_{t_0}^*=m_f$. Note that
$0\le f'\le1$,
$\varphi(f')=f'(x)=f(t_0x)=0$, and  $C(\partial_F\Gamma)$ commutes with
$\mathbb{C}1\bar{\otimes}Q$ in $\mathbb{B}(L^2(\widetilde{M}))$ (both act
diagonally in the two $\ell^2(\Gamma)$-coordinates, the former by scalars on
each fibre). Using the multiplicative-domain identities and $\cP$-centrality, we obtain
\begin{align*}
\widetilde\tau\bigl(a u_{t_0}^*(1\otimes b)\bigr)
 &=\varphi\bigl(a u_{t_0}^*(1\otimes b)\bigr) =\varphi\bigl(a u_{t_0}^*(1\otimes b)m_f\bigr)\\
 &=\varphi\bigl(u_{t_0}^*(1\otimes b)m_f a\bigr) =\varphi\bigl(m_{f'}(1\otimes b)u_{t_0}^*a\bigr) \\&=\varphi(m_{f'})\varphi\bigl((1\otimes b)u_{t_0}^*a\bigr)=0.
\end{align*}
Here the second equality uses $f(x)=1$, the third uses $a\in\cP$,
the fourth uses covariance and commutation with $\mathbb{C}1\bar{\otimes}Q$, and the
last uses $f'(x)=0$.

%Let $\mathbb{E}_Q$ denote the canonical $\tau$-preserving conditional expectation from $M$ onto $Q$. 
Let $\mathbb{E}_Q$ denote the canonical conditional expectation from $M$ onto $\mathbb{C}1\bar\otimes Q$, identified with $Q$.
The last identity gives
$\tau_Q(b\mathbb{E}_Q(au_{t_0}^*))=0$ for every $b\in Q$. Hence
$\mathbb{E}_Q(au_{t_0}^*)=0$.
Since $t_0\notin\operatorname{Rad}(\Gamma)$ was arbitrary, all Fourier coefficients of $a$ outside
$\operatorname{Rad}(\Gamma)$ vanish. Thus, $a\in L(\operatorname{Rad}(\Gamma))\bar\otimes Q$.

Since $a\in\cP$ was arbitrary, this proves the assertion.

\end{proof}
The following is then immediate.
\begin{cor}\label{cor:amenable-tensor}
Let $\Gamma\curvearrowright(X,\nu)$ be a p.m.p.~action. Every amenable
$\Gamma$-invariant von Neumann subalgebra
$\cP\subseteq L(\Gamma)\bar\otimes\Linf$ for the diagonal action satisfies
$\cP\subseteq L(\operatorname{Rad}(\Gamma))\bar\otimes\Linf$.
In particular, if $\operatorname{Rad}(\Gamma)=\{e\}$, then
$\cP\subseteq\mathbb{C}1\bar{\otimes}\Linf$.
\end{cor}

\begin{cor}\label{cor:diagonal-tensor-center}
Let $\Gamma\curvearrowright(X,\nu)$ be a p.m.p.~action and let
$\cP\subseteq L(\Gamma)\bar\otimes\Linf$ be invariant under the diagonal action.
Then $Z(\cP)\subseteq L(\operatorname{Rad}(\Gamma))\bar\otimes\Linf$.
In particular, if $\operatorname{Rad}(\Gamma)=\{e\}$, then
$Z(\cP)\subseteq\mathbb{C}1\bar{\otimes}\Linf$.
\end{cor}
\begin{proof}
Since $Z(\cP)$ is abelian, it is amenable. Since $\cP$ is $\theta$-invariant,
$Z(\cP)$ is also $\theta$-invariant. Apply Corollary~\ref{cor:amenable-tensor}
to $Z(\cP)$.
\end{proof}
Combining these results, we can now prove Theorem~\ref{thm:main-tensor} from the introduction.
\begin{proof}[Proof of Theorem~\ref{thm:main-tensor}]
As the amenable radical of $\Gamma$ is trivial, Corollary~\ref{cor:diagonal-tensor-center} gives $Z(\cP)\subseteq\mathbb{C}1\bar{\otimes}\Linf$. By
Lemma~\ref{lem:tensor-center-valued-character}, $Z(\cP)=\mathbb{C}1\bar{\otimes} B$ for a $\Gamma$-invariant von
Neumann subalgebra $B\subseteq\Linf$. Theorem~\ref{thm:tensor-center-in-A},
including its one-point case, now gives
$\cP=L(N)\bar\otimes B$ for some $N\trianglelefteq\Gamma$.
\end{proof}

The following proposition shows that pointwise mixing does not suffice in
Theorem~\ref{thm:main-tensor}.
\begin{proposition}\label{prop:direct-augmentation}
Let $\F_n$ be the nonabelian free group on $2\leq n\leq\infty$ generators,
and let $V=\mathbb Z\F_n$, regarded as an additive group.
There is a subgroup $H<V\times\F_n$ with the following properties.
\begin{enumerate}
\item $H$ is invariant under $(v,s)\mapsto(gv,gsg^{-1})$ for every $g\in\F_n$;
\item $\{0\}\neq H\cap V\neq V$;
\item $H\neq H_1\times H_2$ for all subgroups $H_1\leq V$ and $H_2\leq\F_n$;
\item in the tensor coordinates $L(V\times\F_n)\cong L(\F_n)\bar\otimes L(V)$,
$L(H)\neq B_2\bar\otimes B_1$ for all von Neumann subalgebras
$B_1\subseteq L(V)=L^\infty(\mathbb T^{\F_n})$ and $B_2\subseteq L(\F_n)$.
\end{enumerate}
\end{proposition}
\begin{proof}
Let $\epsilon$ be the augmentation homomorphism from $V$ to $\mathbb Z$, given by
$\epsilon(\sum_s c_ss)=\sum_s c_s$, and let $I=\ker\epsilon$.
Choose a free generator $s_1$ and let $\phi$ be the homomorphism from $\F_n$ to $\mathbb Z$ that sends
$s_1$ to $1$ and every other free generator to $0$. Define the homomorphism $D$ from $V\times\F_n$ to $\mathbb Z$ by $D(v,s)=\epsilon(v)-\phi(s)$. Its kernel is $H=\{(\phi(s)e+x,s)\mid x\in I,s\in\F_n\}$,
where $e$ denotes the basis element of $V$ at the group identity.
Since $\epsilon(gv)=\epsilon(v)$ and
$\phi(gsg^{-1})=\phi(s)$, the kernel is invariant under the given
diagonal action. Its intersection with $V$ is $I$, which is nonzero
because $s_1-e\in I$ and is proper because $\epsilon(e)=1$.

The projection of $H$ to $\F_n$ is onto. Its projection to $V$ is also
onto, since for $v\in V$ we can choose $s$ with $\phi(s)=\epsilon(v)$.
But $(0,s_1)\notin H$. If $H=H_1\times H_2$, the two projections would
force $H_1=V$ and $H_2=\F_n$, contradicting this last fact. This proves (3).

Suppose that $L(H)=B_2\bar\otimes B_1$. For each $(v,s)\in H$, the unitary
$\lambda_s\otimes\lambda_v$ belongs to $B_2\bar\otimes B_1$.
Applying the normal slice maps defined by
$z\mapsto\tau_{\F_n}(\lambda_s^*z)$ and
$z\mapsto\tau_V(\lambda_v^*z)$ gives $\lambda_v\in B_1$ and $\lambda_s\in B_2$.
Surjectivity of the two group projections therefore implies
$B_1=L(V)$ and $B_2=L(\F_n)$. It follows that
$\lambda_{(0,s_1)}\in L(H)$, which is impossible because the Fourier expansion of every
element of $L(H)$ is supported in $H$. This proves (4).
\end{proof}

The action on $L(V)=L^\infty(\mathbb T^{\F_n})$ is the Bernoulli action
with product Haar measure. For every $g\neq e$, the powers of $g$
eventually move any finite set of coordinates off any other finite set.
Thus, $\sigma_g$ is mixing, first for functions depending on finitely
many coordinates and then for all $L^2$ functions by density.
With $I$ and $e$ as in the proof, the action is not weakly mixing relative
to the proper $\F_n$-factor $L(I)$. Indeed, $\mathbb{E}_{L(I)}(\lambda_e)=0$, whereas
$\lambda_e^*\sigma_g^j(\lambda_e)=\lambda_{g^j-e}\in L(I)$ has $L^2$ norm
one for every $g\neq e$ and $j\geq1$.

The map $(x,s)\mapsto(x+\phi(s)e,s)$ is an isomorphism from
$I\times\F_n$ onto $H$. Thus, $L(H)\cong L(\F_n)\bar\otimes L(I)$
abstractly, and the failure in the example concerns the given tensor
coordinates.

\section{Examples of weakly mixing actions}\label{sec:prime}
We construct prime and non-prime actions satisfying the hypotheses of
both Theorems~\ref{thm:main} and~\ref{thm:main-tensor}.
Throughout this section, $(X,\nu)$ is a standard nonatomic probability
space. Recall that a p.m.p.\ action is \emph{prime} if its only
factors
are the trivial factor and the whole system; here a \emph{factor} of an action $\rho$ of $G$ on $(X,\nu)$ means a $\rho$-invariant
$\sigma$-subalgebra of $\MALG_\nu$ (the measure algebra of $(X,\nu)$), and equality of factors, like every
other equality below, is understood modulo null sets. By \emph{atoms} of a Boolean algebra we mean its nonempty minimal members.

Note that an action of a free group is pointwise weakly mixing if and only if it is \emph{totally
weakly mixing} in the sense of \cite{kechris2010global}. %if its restriction to every infinite subgroup is weakly
%mixing. For a free group, total weak mixing can be tested element by
%element.

\begin{comment}
\begin{lem}\label{lem:twm-elementwise}
Let $2\leq n\leq\infty$ and let $\rho$ be a p.m.p.\ action of
$\F_n$ on $(X,\nu)$. Then $\rho$ is
totally weakly mixing if and only if $\rho(g)$ is a weakly mixing
transformation for every $g\in\F_n\setminus\{e\}$.
\end{lem}

\begin{proof}
If $\rho$ is totally weakly mixing and $g\neq e$, then $\langle g\rangle$
is infinite because $\F_n$ is torsion-free. Hence $\rho|_{\langle g\rangle}$
is weakly mixing, which means that the transformation $\rho(g)$ is weakly
mixing.

Conversely, assume that $\rho(g)$ is weakly mixing for every $g\neq e$, and
let $H\leq\F_n$ be infinite. If $\rho|_H$ were not weakly mixing, then the
Koopman representation of $H$ on $L^2_0(X,\nu)$ would contain a nonzero
finite-dimensional invariant subspace. Choosing $h\in H\setminus\{e\}$,
this subspace would be invariant under the operator induced by $\rho(h)$,
hence, being finite-dimensional, it would contain an eigenvector of
that operator, contradicting weak mixing of the transformation $\rho(h)$.
\end{proof}

\end{comment}

We shall first construct an essentially free, prime, totally weakly mixing action of
$\F_2$, adopting the arguments appearing in \cite[Section 17]{kechris2010global} and in \cite{wrobel2026orbit}. The actions of all the other free groups will then be obtained by
restriction to suitable free subgroups of $\F_2$ containing the first
generator.

\subsection{A generic second generator}

We begin with the precise version of the generic-freeness argument needed
below. We follow the argument used in
\cite[proof of~17.2, and Lemmas~17.3 and~17.4]{kechris2010global}
and adapt it to our purposes. It goes through as it is; we add the details
for completeness and clarity nonetheless.
The two lemmas of Kechris are recorded below as
Lemmas~\ref{lem:disjointization} and~\ref{lem:local-modification}, the
second in a slightly strengthened form.

Let $\MALG_\nu$ be the measure algebra of $(X,\nu)$, endowed with the metric
$d(A,B)=\nu(A\triangle B)$. The weak topology on $\Aut(X,\nu)$ is the
smallest topology for which, for every $A\in\MALG_\nu$, the map
$R\mapsto R(A)$ is continuous. Thus, a basic weak neighbourhood of
$R_0\in\Aut(X,\nu)$ has the form
\begin{equation*}
 \cN(R_0;C_1,\ldots,C_q;\kappa)
 =\bigl\{R\mid \nu(R(C_j)\triangle R_0(C_j))<\kappa,
                  1\leq j\leq q\bigr\},
\end{equation*}
where $C_1,\ldots,C_q\in\MALG_\nu$ and $\kappa>0$. With this topology,
$\Aut(X,\nu)$ is a Polish topological group, and the sets above do form
a basis for it; see \cite[Section~1(B)]{kechris2010global}. When we talk about open sets inside $\Aut(X,\nu)$ we will always mean the weak topology. We shall also occasionally use the
uniform metric
 $\du(R,S)=\nu\bigl(\{x\in X\mid R(x)\neq S(x)\}\bigr)$.
All equalities between transformations and measurable sets are understood
modulo null sets. When a pointwise argument is used, Borel representatives
are chosen and one common null set is discarded.
For $R\in\Aut(X,\nu)$, we record the identity
\begin{equation}\label{eq:du-fix}
 \du(R,\id)=\nu(\text{supp}(R))=1-\nu(\text{Fix}(R)),
\end{equation}
for which see \cite[Section~1(C), p.~3]{kechris2010global}; thus,
$\du(R,\id)=1$ if and only if $\nu(\text{Fix}(R))=\nobreak0$.

The following is Kechris's disjointization lemma, written in the form that
will be used below.

\begin{lem}
\label{lem:disjointization}
Let $P_1,\ldots,P_k\in\Aut(X,\nu)$, and let $D\subseteq X$ be a Borel set
such that $P_1(x),\ldots,P_k(x)$ are pairwise distinct for every $x\in D$.
If $0\leq a<\nu(D)$, then there are pairwise disjoint Borel sets
$A_1,\ldots,A_M\subseteq D$ such that
$\nu(A_1\cup\cdots\cup A_M)>a$ and, for every $r=1, \ldots,M$, the sets
$P_1(A_r),\ldots,P_k(A_r)$ are pairwise disjoint.
\end{lem}

\begin{proof}
Because $X$ is a standard Borel space, there is a countable Boolean algebra
$\mathcal A$ of Borel sets which separates points. For example, after
identifying $X$ with a Borel subset of $[0,1]$, one may pull back the algebra
of finite unions of intervals with rational endpoints.

Fix $x\in D$. Since $P_1(x),\ldots,P_k(x)$ are distinct, for every pair of
these points choose a member of $\mathcal A$ which contains one and not the
other. The finitely many chosen sets generate a finite Boolean algebra, and
the points $P_1(x),\ldots,P_k(x)$ lie in distinct atoms of that finite
algebra. %Since $\mathcal A$ is a Boolean algebra, those atoms belong to
%$\mathcal A$. 
Hence there exist pairwise disjoint sets
$V_1,\ldots,V_k\in\mathcal A$ with $P_i(x)\in V_i$ for every $i=1, \ldots,k$.

For every pairwise disjoint $k$-tuple
$\mathbf V=(V_1,\ldots,V_k)$ from $\mathcal A$, let
 $U_{\mathbf V}=D\cap\bigcap_{i=1}^k P_i^{-1}(V_i)$.
There are only countably many such tuples, these sets cover $D$, and
$P_i(U_{\mathbf V})\subseteq V_i$; therefore the images of each
$U_{\mathbf V}$ are pairwise disjoint Borel sets. Enumerate the nonempty sets in this
cover as $U_1,U_2,\ldots$ and let $B_1=U_1$ and
$B_r=U_r\setminus\bigcup_{j<r}U_j$ for $r\geq2$.
Then $D=\bigsqcup_{r\geq1}B_r$, and the images
$P_1(B_r),\ldots,P_k(B_r)$ are pairwise disjoint for every $r$. Since
$\nu(\bigcup_{r=1}^N B_r)\to\nu(D)>a$, a sufficiently large finite initial
union has measure greater than $a$. Discarding its null members gives the
required sets.
\end{proof}

For a measurable set $E$, when we say that an automorphism $P\in \Aut(X,\nu)$
\emph{preserves} $E$ we mean that $P(E)=E$ modulo null sets. We also define the \emph{support} of $P$ as $\text{supp}(P)=\{x\in X:P(x)\neq x\}$.

\begin{lem}\label{lem:local-modification}
Let $R\in\Aut(X,\nu)$, let $D\subseteq X$ be Borel, and let
$\mathcal C$ be a finite family of Borel sets. There is an involution
$P\in\Aut(X,\nu)$ such that
\begin{enumerate}[label=\textup{(\roman*)}]
\item $\text{supp}(P)=D$ modulo null sets;
\item $P$ preserves every member of the finite Boolean algebra generated by
\[
 \{D\}\cup\mathcal C
 \cup\{R(C)\mid C\in\mathcal C\}
 \cup\{R^{-1}(C)\mid C\in\mathcal C\}.
\]
\end{enumerate}
Consequently, for both $R_+=RP$ and $R_-=PR$ one has
 $R_\pm(C)=R(C)$,
 $R_\pm^{-1}(C)=R^{-1}(C)$ for every $C\in\mathcal C$.
\end{lem}

\begin{proof}
The displayed family is finite, hence it generates a finite Boolean algebra.
Atoms of this algebra arise as intersections obtained by choosing, for each generator, either the set itself or its complement. Since $D$ is one of the generators, every atom
is either contained in $D$ or disjoint from $D$.

Let $E$ be a positive-measure atom contained in $D$. The restriction of
$\nu$ to $E$ is nonatomic, so we can choose a Borel partition
$E=E^0\sqcup E^1$ with $\nu(E^0)=\nu(E^1)$. The two standard nonatomic
finite measure spaces $E^0$ and $E^1$ are measure-isomorphic; this follows,
for example, from the standard isomorphism theorem recalled in
\cite[Section~1(A)]{kechris2010global}. Choose a measure-preserving Borel
isomorphism $\psi_E$ from $E^0$ to $E^1$ modulo null sets. Define $P$ to be
$\psi_E$ on each $E^0$, $\psi_E^{-1}$ on each $E^1$, and the identity outside the
union of the positive-measure atoms contained in $D$. Then $P$ is a
measure-preserving involution which has no fixed point on $D$ modulo null
sets. % because $E^0$ and $E^1$ are disjoint and, $D$ being a union of
%generators' atoms, $D$ coincides modulo null sets with the union of the
%positive-measure atoms contained in it, 
Furthermore, it preserves every atom, hence
every member, of the finite Boolean algebra in question. In particular, for
$C\in\mathcal C$, we have
$ P(C)=C$, $P(R(C))=R(C)$, $P(R^{-1}(C))=R^{-1}(C)$.
Hence $R_+(C)=R(P(C))=R(C)$ and
$R_+^{-1}(C)=P(R^{-1}(C))=R^{-1}(C)$, and likewise
$R_-(C)=P(R(C))=R(C)$ and $R_-^{-1}(C)=R^{-1}(P(C))=R^{-1}(C)$.
The claim follows.
\end{proof}

\begin{remark}
The sets $R(C)$, $C\in\mathcal C$, are not needed for the conclusion about
$R_+=RP$, and are absent from
\cite[Lemma~17.4]{kechris2010global}; they are included here because the
conclusion about $R_-=PR$ is used in the case $\varepsilon=-1$ of
Proposition~\ref{prop:remove-fixed-points}.
\end{remark}

In what follows we shall consider p.m.p.\ actions of $\F_2$. Denote the generators of $\F_2$ by $a,b$. Given $S,T \in \Aut(X,\nu)$, denote by $\alpha^{T,S}$ the action of $\F_2$ on $(X,\nu)$ determined by the equalities $\alpha^{T,S}(a) = T$, $\alpha^{T,S}(b) = S$. If $v \in \F_2$, we will simply write $v^{T,S}:=\alpha^{T,S}_v \in \Aut(X,\nu)$.

%We use the usual convention for products of transformations. If
%$v=\ell_1\cdots\ell_n$ is a word, then
%$v^S=\ell_1^S\circ\cdots\circ\ell_n^S$, so the rightmost letter acts first.
If $n \in \N$ and $v=\ell_1\cdots\ell_n\in\F_2$ is a reduced word of length $n$, define the terminal subwords of $v$ by $v_0=e$ and
$v_i=\ell_{n-i+1}\cdots\ell_n$ for $1\leq i\leq n$.
%Thus in the situation $v_i^S(x)$ is the point reached after the first $i$ letters, read from
%the right, have acted on $x$. %Powers such as $b^k$ are expanded into
%$|k|$ consecutive letters $b$ or $b^{-1}$.
The conventions above  are in force throughout the section. The next claim is the only local surgery needed in the proof. We shall use without any further comment that the elements of $\Aut(X,\nu)$ are injective (modulo null sets).

\begin{proposition}
\label{prop:remove-fixed-points}
Fix an aperiodic transformation $T\in\Aut(X,\nu)$ and let $S\in\Aut(X,\nu)$. 
%Let $\F_2=\langle a,b\rangle$, let $S\in\Aut(X,\nu)$, and evaluate words by
%$a\mapsto T$, $b\mapsto S$. 
Let $n \in \N$ and let $v\in \F_2$ be a reduced word of length $n$ in
which $b^{\pm1}$ occurs. Suppose that $A\subseteq X$ is Borel and
$v_0^{T,S}(A),v_1^{T,S}(A),\ldots,v_{n-1}^{T,S}(A)$ are pairwise disjoint. Given a finite family $\mathcal C$ of Borel subsets of $X$,
there is $S'\in\Aut(X,\nu)$ such that
\begin{enumerate}[label=\textup{(\roman*)}]
\item $S'^{\pm1}(C)=S^{\pm1}(C)$ for every $C\in\mathcal C$;
\item $\nu(A\cap\text{Fix}(v^{T,S'}))=0$.
\end{enumerate}
\end{proposition}

\begin{proof}
Choose a maximal $i_0\in \{0, \ldots,n-1\}$ such that
$ v_{i_0+1}=b^\varepsilon v_{i_0}$ for some $\varepsilon\in\{1,-1\}$.
This corresponds to the leftmost occurrence of $b^{\pm1}$ in $v$.
%$v=\ell_1\cdots\ell_n$. 
Consequently, every letter which appears before this
selected occurrence is $a$ or $a^{-1}$. Let
 $F=A\cap\text{Fix}(v^{T,S})$, and
 $D=v_{i_0}^{T,S}(F)$.
 
If $\nu(F)=0$, it suffices to take $S'=S$. Suppose that $\nu(F)>0$, and apply
Lemma~\ref{lem:local-modification} to $S$, $D$, and $\mathcal C$. Let $P$
be the resulting involution with support $D$.
Note that $D\subseteq v_{i_0}^{T,S}(A)$, so that $D$ is disjoint from
$v_i^{T,S}(A)$ for every $i<n$ with $i\neq i_0$.

Suppose first that $\varepsilon=1$ and set $S'=SP$. We first claim that,
for $x\in A$ and $0\leq i\leq i_0$,
\begin{equation}\label{eq:path-before-selected-positive}
 v_i^{T,S'}(x)=v_i^{T,S}(x).
\end{equation}
We argue by induction on $i$. For $i=0$ both sides equal $x$. Assume
the equality holds at some $i<i_0$,
and let $z=v_i^{T,S'}(x)=v_i^{T,S}(x)$. If the next letter appearing on the left is $a^{\pm1}$, the
next step is unchanged. If it is $b$, then
$S'(z)=S(Pz)=S(z)$ because
$z\in v_i^{T,S}(A)$ and this set is disjoint from
$D\subseteq v_{i_0}^{T,S}(A)$. If the next letter is $b^{-1}$, then
$v_{i+1}^{T,S'}(x)=S'^{-1}(z)=P(S^{-1}z)=P(v_{i+1}^{T,S}(x))$.
When $i+1<i_0$, the point $v_{i+1}^{T,S}(x)$ lies outside $D$ by the same
pairwise disjointness. The only remaining possibility is $i+1=i_0$; but
then the letter immediately to the right of the selected $b$ would be
$b^{-1}$, contradicting the fact that $v$ is reduced. Hence the last expression is
$S^{-1}z$, and equation~\eqref{eq:path-before-selected-positive} follows.

Write $v=qb^{\varepsilon}v_{i_0}$, where
$q=\ell_1\cdots\ell_{n-i_0-1}$. By the maximality of $i_0$, the word $q$
involves only letters $a^{\pm1}$, so that $q^{T,S'}=q^{T,S}=T^{m}$ for some $m\in\Z$;
in particular this transformation does not depend on the second generator.
%and it is injective.

Take $x\in F$ and let $y=v_{i_0}^{T,S}(x)$. Note that $y\in D$, so that
$P(y)\neq y$. Since $S$ is injective, it follows that
$S'(y)=S(P(y))\neq S(y)$. By equation~\eqref{eq:path-before-selected-positive} we
have $v_{i_0}^{T,S'}(x)=y$, and therefore
$v^{T,S'}(x) =q^{T,S}\bigl(S'(y)\bigr) \neq q^{T,S}\bigl(S(y)\bigr) =v^{T,S}(x) =x$.

If $x\in A\setminus F$, then $v_{i_0}^{T,S}(x)\notin D$, since
$D=v_{i_0}^{T,S}(F)$. Hence
$P\bigl(v_{i_0}^{T,S}(x)\bigr)=v_{i_0}^{T,S}(x)$.
By equation~\eqref{eq:path-before-selected-positive},
$v_{i_0}^{T,S'}(x)=v_{i_0}^{T,S}(x)$, and therefore
$S'\bigl(v_{i_0}^{T,S'}(x)\bigr) = S\bigl(P(v_{i_0}^{T,S}(x))\bigr) = S\bigl(v_{i_0}^{T,S}(x)\bigr)$.
Since the remaining letters of $v$ are those of $q$, it follows that
$v^{T,S'}(x)=v^{T,S}(x)\neq x$. This proves (ii) when $\varepsilon=1$.

The case $\varepsilon=-1$ reduces to the case $\varepsilon=1$ by replacing
the second generator $b$ with $b^{-1}$ and $S$ with $S^{-1}$. More precisely,
let $\sigma$ be the automorphism of $\F_2$ determined by setting $\sigma(a)=a$ and
$\sigma(b)=b^{-1}$, and let $\widehat S=S^{-1}$ and $w=\sigma(v)$. Then $w^{T,\widehat{S}}=v^{T,S}$ and
$w_i^{T,\widehat{S}}=v_i^{T,S}$ for $0\leq i < n$.
As $w$ is a reduced word of length $n$ in which $b^{\pm1}$ occurs, the
hypothesis on $A$ holds for $w$ and $\widehat S$. The leftmost occurrence
of $b^{\pm1}$ in $w$ is at the same place as in $v$ and has exponent $+1$,
and the sets $F$ and $D$ attached to $w$ and $\widehat S$ are the same as the ones
attached to $v$ and $S$. Finally, the finite family displayed in
Lemma~\ref{lem:local-modification} is the same for $\widehat S$ as for $S$,
because $\widehat S^{\pm1}(C)=S^{\mp1}(C)$; so the same involution $P$
occurs. The case already treated yields
$\widehat S'=\widehat S P=S^{-1}P$ with
$\nu\bigl(A\cap\text{Fix}(w^{T,\widehat S'})\bigr)=0$. Setting
 $S'=(\widehat S')^{-1}=PS$,
and noting that $w^{T,\widehat S'}=v^{T,S'}$ for the same reason as above, we
obtain $S'^{\pm1}(C)=S^{\pm1}(C)$ for every $C\in\mathcal C$ and
$\nu\bigl(A\cap\text{Fix}(v^{T,S'})\bigr)=0$.
Assertion (i) follows from the final claim of
Lemma~\ref{lem:local-modification}, applied to $R_+=SP$ when
$\varepsilon=1$ and to $R_-=PS$ when $\varepsilon=-1$.

\end{proof}

We shall use repeatedly the following
observation, whose hypothesis is exactly what condition~(i) of
Proposition~\ref{prop:remove-fixed-points} provides.
Let $v \in \F_2$ be a reduced word of length $n$, let $S, S', T \in \Aut(X,\nu)$, and let $B\subseteq X$. Set
$B_i=v_i^{T,S}(B)$ for $i=0, \ldots, n$. If
 $S'^{\pm1}(B_i)=S^{\pm1}(B_i)$ for $0\leq i<n$,
then an induction over the successive letters of $v$ gives, for $0\leq i\leq n$,
\begin{equation}\label{eq:path-sets-preserved}
 v_i^{T,S'}(B)=v_i^{T,S}(B).
\end{equation}
For an $a^{\pm1}$-letter there is nothing to check; for a $b$-letter use
$S'(B_i)=S(B_i)$, and for a $b^{-1}$-letter use
$S'^{-1}(B_i)=S^{-1}(B_i)$.

\begin{proposition}
\label{prop:generic-second-generator}
Let $T\in\Aut(X,\nu)$ be aperiodic.
%let
%$\alpha^S$ be the action of $\F_2$ on $(X,\nu)$ determined by
% $\alpha^S(a)=T$, and
% $\alpha^S(b)=S$.
Then the set of $S$ for which $\alpha^{T,S}$ is essentially free is comeager in
$\Aut(X,\nu)$. If $T$ is ergodic, every such $\alpha^{T,S}$ is ergodic.
\end{proposition}

\begin{proof}
The ergodicity assertion is immediate since every $\alpha^{T,S}$-invariant set is
$T$-invariant.

For a non-trivial reduced word $v\in\F_2$ and $0<\epsilon\leq1$, let
$\cU(v,\epsilon)=\{S\in\Aut(X,\nu)\mid \du(v^{T,S},\id)>1-\epsilon\}$.
By equation~\eqref{eq:du-fix}, $S$ belongs to $\cU(v,\epsilon)$ for every
$\epsilon$ if and only if $\nu(\text{Fix}(v^{T,S}))=0$.

We first prove that the sets $\cU(v,\epsilon)$ are open.
The word map $S\mapsto v^{T,S}$ is  continuous, since
$\Aut(X,\nu)$ is a topological group. Moreover, the map
$R\longmapsto \du(R,\id)$
is lower semicontinuous; equivalently, for every
$c>0$, the set
$\{R\in\Aut(X,\nu)\mid \du(R,\id)>c\}$
is  open; see \cite[Section~1(C), pp.~3--4]{kechris2010global}.
Consequently, $\cU(v,\epsilon)$ is  open.

We now prove by induction on the  length of $v$  that
$\cU(v,\epsilon)$ is dense
for every non-trivial word $v$ and every $0<\epsilon\leq1$.

If $|v|=1$ and $v=a^{\pm1}$, then
$\du(v^{T,S},\id)=1$ for every $S$ because $T$ is aperiodic, by
equation~\eqref{eq:du-fix}; hence $\cU(v,\epsilon)=\Aut(X,\nu)$. If
$v=b^{\pm1}$, density follows from the fact that aperiodic transformations
form a dense $G_\delta$ subset of $\Aut(X,\nu)$
\cite[Proposition~2.3]{kechris2010global}, since
$\du(S^{\pm1},\id)=1$ for every aperiodic $S$.

Now fix $n\geq2$, assume the density assertion for every non-trivial  word
of length less than $n$, and let $v\in \F_2$ have length $n$. Let
$\cN$ be an arbitrary nonempty basic 
neighbourhood, so that 
\begin{equation} \label{def:N}\cN=\bigl\{R\in\Aut(X,\nu)\mid 
   \nu\bigl(R(C_j)\triangle S_*(C_j)\bigr)<\kappa,1\leq j\leq q\bigr\}
\end{equation}
for some $S_*\in \Aut(X, \nu)$, $q \in \N$, Borel sets $C_1, \ldots,C_q$ in $X$ and $\kappa>0$. 
Since such sets form a basis for the weak topology, it
suffices to show that $\cU(v,\epsilon)\cap\cN\neq\varnothing$.
If $v\in\langle a\rangle$,
then $v^{T,S}$ is a nonzero power of the aperiodic transformation $T$ and has
null fixed-point set for every $S$. We may therefore assume that
$b^{\pm1}$ occurs in $v$.

By the induction hypothesis,
\[
 \cD_{<n}
 =\bigcap_{\substack{u\neq e\\|u|<n}}
    \bigcap_{r\geq1}\cU(u,1/r)
\]
is a dense $G_\delta$ set because it is a countable intersection of dense open
sets, and $\Aut(X,\nu)$ is Polish, hence a Baire space. Choose
$S_0\in\cN\cap\cD_{<n}$, which is possible because $\cN$ is nonempty
and open. By equation~\eqref{eq:du-fix}, for every non-trivial word $u$
with $|u|<n$,
\begin{equation}\label{eq:Dn-fix}
 \nu\bigl(\text{Fix}(u^{{T,S_0}})\bigr)=0.
\end{equation}

Write $v=\ell_1\cdots\ell_n$ and for $i=0, \ldots,n$ let $v_i$ be the terminal subwords
of $v$, as introduced before Proposition \ref{prop:remove-fixed-points}. We
claim that
\begin{equation}\label{eq:terminal-points}
 v_0^{T,S_0}(x),v_1^{T,S_0}(x),\ldots,v_{n-1}^{T,S_0}(x)
\end{equation}
are pairwise distinct for almost every $x\in X$. Indeed, if $0\leq i<j<n$, then
$v_j=u_{j,i}v_i$, where
$u_{j,i}=\ell_{n-j+1}\cdots\ell_{n-i}$ is a non-trivial word of
length $j-i<n$. If
$v_j^{T,S_0}(x)=v_i^{T,S_0}(x)$ and
$y=v_i^{T,S_0}(x)$, then $u_{j,i}^{T,S_0}(y)=y$. The fixed-point set of
$u_{j,i}^{T,S_0}$ is null by equation~\eqref{eq:Dn-fix}; since
$v_i^{T,S_0}$ preserves measure, the set of such $x$ is null. There are only
finitely many pairs $(i,j)$, proving the claim.

Let $D_0\subseteq X$ be a conull Borel set on which the points in
equation~\eqref{eq:terminal-points} are pairwise distinct. Apply
Lemma~\ref{lem:disjointization} to
$v_0^{T,S_0},\ldots,v_{n-1}^{T,S_0}$ on $D_0$, with $a=1-\epsilon$. We obtain
pairwise disjoint Borel sets $A_1,\ldots,A_M$ such that
$\nu(A_1\cup\cdots\cup A_M)>1-\epsilon$ and, for each $r=1, \ldots, M$, the sets
$v_0^{T,S_0}(A_r),\ldots,v_{n-1}^{T,S_0}(A_r)$ are pairwise disjoint. Let
$\eta=\nu(A_1\cup\cdots\cup A_M)-(1-\epsilon)>0$ and
$\delta=\frac{\eta}{2M}$.

We shall now show that for each $r=1,\ldots,M$ there exist
$S_r\in \Aut(X,\nu)$ and a finite collection
$\cB_r$ of pairwise disjoint Borel subsets of $A_r$ such that
the following hold:
\begin{enumerate}[label=\textup{(\alph*)}]
\item $S_r\in\cN$;
\item for every $p=1, \ldots, M, p\leq r$ and every $B\in\cB_p$,
      $B\cap v^{T,S_r}(B)=\varnothing$;
\item for every $j=1,\ldots, M$, $j>r$ and $i=0,\ldots,n-1$,
      $v_i^{T,S_r}(A_j)=v_i^{T,S_0}(A_j)$;
\item $\displaystyle\sum_{B\in\cB_p}\nu(B)\geq\nu(A_p)-\delta$
      for every $p=1, \ldots, M, p\leq r$.
\end{enumerate}
We construct these inductively, starting at $r=0$ with the chosen $S_0$ and the empty family $\cB_0$.

Now let $r\in \{1,\ldots,M\}$, and suppose that $S_{r-1}$ and $\cB_1,\ldots,\cB_{r-1}$ have been constructed
so that (a)--(d) hold at stage $r-1$.
From (c) at stage $r-1$, applied to $j=r$, we see that
$v_i^{T,S_{r-1}}(A_r)=v_i^{T,S_0}(A_r)$
for $0\leq i<n$,
so these are the original pairwise disjoint images of $A_r$; in
particular $S_{r-1}$, $A_r$ and $v$ satisfy the hypothesis of
Proposition~\ref{prop:remove-fixed-points}.

The next perturbation must preserve exactly three finite collections of
sets, namely the sets defining $\cN$, the complete $v$-paths of the blocks
already chosen, and the complete $v$-paths of the sets which have not yet
been seen. Accordingly, consider the finite family
\[
\begin{aligned}
\mathcal C_r={}&\{C_1,\ldots,C_q\} \cup\\
&\{v_i^{T,S_{r-1}}(B)\mid p<r,B\in\cB_p,0\leq i<n\} \cup\\
&\{v_i^{T,S_{r-1}}(A_j)\mid j>r,0\leq i<n\}.
\end{aligned}
\]
%This family is finite because $q,M,n$ and every $\cB_p$ are finite.

Apply Proposition~\ref{prop:remove-fixed-points} to $S_{r-1}$, $A_r$, and
$\mathcal C_r$, and call the resulting transformation $S_r$;
thus, $S_r^{\pm1}(C)=S_{r-1}^{\pm1}(C)$ for every $C\in\mathcal C_r$,
and $\nu\bigl(A_r\cap\text{Fix}(v^{T,S_r})\bigr)=0$. 
If $\nu(A_r)\leq\delta$, let $\cB_r=\varnothing$. If
$\nu(A_r)>\delta$, apply Lemma~\ref{lem:disjointization} to
$\id$ and $v^{T,S_r}$ on the Borel set
$A_r\setminus\text{Fix}(v^{T,S_r})$, with $a=\nu(A_r)-\delta$.
Note that the conditions of the lemma are satisfied as
$\nu\bigl(A_r\setminus\text{Fix}(v^{T,S_r})\bigr)=\nu(A_r)>\nu(A_r)-\delta$,
and $\id(x)\neq v^{T,S_r}(x)$ for $x \in A_r\setminus\text{Fix}(v^{T,S_r})$.
We thus obtain a finite
family $\cB_r$ of pairwise disjoint subsets of $A_r$ such that
$\sum_{B\in\cB_r}\nu(B)>\nu(A_r)-\delta$ and
$B\cap v^{T,S_r}(B)=\varnothing$ for every $B\in\cB_r$.

We shall now check that the conditions
(a)--(d) hold at stage $r$.

First check (a). Recall equation~\eqref{def:N}.
%\[\cN=\bigl\{R\in\Aut(X,\nu)\mid 
%   \nu\bigl(R(C_j)\triangle S_*(C_j)\bigr)<\kappa,1\leq j\leq q\bigr\}.
%\]
It is obvious that if $R,R'\in\Aut(X,\nu)$ satisfy
$R'(C_j)=R(C_j)$ for $1\leq j\leq q$, then $R\in\cN$ if and only if
$R'\in\cN$. Now $C_j\in\mathcal C_r$ for $1\leq j\leq q$, by the definition of
$\mathcal C_r$, and $S_r$ was obtained by applying
Proposition~\ref{prop:remove-fixed-points} to $S_{r-1}$ with the family
$\mathcal C_r$. Condition~(i) of that proposition, applied to $C=C_j$,
gives
$ S_r(C_j)=S_{r-1}(C_j)$ for $1\leq j\leq q$.
The induction hypothesis (a) at stage $r-1$ says that $S_{r-1}\in\cN$.
By the previous paragraph, $S_r\in\cN$, which is (a) at stage $r$.

To verify (b), let $p<r$ and $B\in\cB_p$. The sets $v_i^{T,S_{r-1}}(B)$, $0\leq i<n$,
belong to $\mathcal C_r$, so equation~\eqref{eq:path-sets-preserved} gives
$v_i^{T,S_r}(B)=v_i^{T,S_{r-1}}(B)$ for $0\leq i\leq n$; taking $i=n$ yields
$v^{T,S_r}(B)=v^{T,S_{r-1}}(B)$, and (b) at stage $r-1$ gives
$B\cap v^{T,S_r}(B)=\varnothing$. For $p=r$ the assertion holds by the properties of $\cB_r$ stated above.

For (c), let $j>r$. The sets $v_i^{T,S_{r-1}}(A_j)$, $0\leq i<n$, also belong to
$\mathcal C_r$, so equation~\eqref{eq:path-sets-preserved}, together with (c) at
stage $r-1$, gives
$v_i^{T,S_r}(A_j)=v_i^{T,S_{r-1}}(A_j)=v_i^{T,S_0}(A_j)$ for $0\leq i<n$.

Finally, the families $\cB_p$ with $p<r$ are unchanged, so (d) persists for
those $p$; and for $p=r$ this holds by the choice of $\cB_r$.

This completes the inductive construction.

Note now that every point in every $B\in\cB_r$ is moved by
$v^{T,S_M}$, because $B\cap v^{T,S_M}(B)=\varnothing$
by (b) at stage $M$. Since the sets $A_r$ are
pairwise disjoint, all the sets in the family $\cB_1 \cup \cdots \cup \cB_M$ are pairwise
disjoint. Therefore
\begin{align*}
 \nu\bigl(\{x\mid v^{T,S_M}(x)\neq x\}\bigr)
 &\geq\sum_{r=1}^M\sum_{B\in\cB_r}\nu(B)
 \geq\sum_{r=1}^M(\nu(A_r)-\delta)\\&=\nu(A_1\cup\cdots\cup A_M)-M\delta
 =1-\epsilon+\frac{\eta}{2}>1-\epsilon.
\end{align*}
Hence $S_M\in\cN\cap\cU(v,\epsilon)$. Since $\cN$ was arbitrary,
$\cU(v,\epsilon)$ is dense, and the induction on $|v|$ is complete.

Thus, every set $\cU(v,1/r)$, with $v\neq e$ and $r\in \N$, is open and
dense in $\Aut(X, \nu)$. Their intersection is comeager. If $S$ belongs to this intersection,
then $\du(v^{T,S},\id)=1$ for every $v\neq e$. %equivalently
%$\nu(\text{Fix}(v^S))=0$ by equation~\eqref{eq:du-fix}. 
Hence $\alpha^{T,S}$ is essentially
free.
\end{proof}

\subsection{The construction for \texorpdfstring{$\F_2$}{F2}}

Note that the group $\Aut(X,\nu)$ contains elements which are at the same time prime and mixing, which will be used below.
Indeed, one may take a mixing rank-one transformation. Explicit examples of these were constructed in \cite{adams1998smorodinsky},
and mixing rank-one transformations have minimal self-joinings and
hence are prime by \cite{ryzhikov1993mixing}.

\begin{proposition}\label{prop:F2-prime-twm}
Let $T\in\Aut(X,\nu)$ be a mixing and prime transformation. There exists a pointwise weakly mixing p.m.p.\ action
$\rho$ of $\F_2=\langle a,b\rangle$ on $(X,\nu)$
with $\rho(a)=T$.
\end{proposition}

\begin{proof}
%Write $\F_2=\langle a,b\rangle$. 
Since $T$ is mixing, it is
aperiodic, so
Proposition~\ref{prop:generic-second-generator} gives
$S\in\Aut(X,\nu)$ such that $\alpha := \alpha^{T,S}$ is an essentially free p.m.p.\ action of $\F_2$ on $(X,\nu)$. This
action is ergodic because its restriction to $\langle a\rangle$ is mixing.

%Thus $\alpha$ is a free ergodic action whose first generator acts by the
%prescribed transformation $T$. 
%This is precisely the situation in which the
%genericity argument of \cite{wrobel2026orbit} operates. 
Note that the proof of
\cite[Theorem~1.1]{wrobel2026orbit} begins by replacing the given essentially free
ergodic action by an orbit equivalent one whose first generator is mixing;
it is only this replacement that we must avoid, since we insist on the
prescribed $T$, and everything after it applies verbatim to $\alpha$. This
is the sole reason why
Proposition~\ref{prop:generic-second-generator} was needed.

Let $[T]$ be the full group of the orbit equivalence relation generated by
$T$, a Polish group in the uniform metric
\cite[Proposition~3.2]{kechris2010global} and hence a Baire space. For
$\phi\in[T]$, set
$ \beta_\phi = \alpha^{T, S \phi}$.
%where composition is read from right to left. 
By \cite[Lemma~2.2]{wrobel2026orbit}, every $\beta_\phi$ is essentially free
and generates the same orbit equivalence relation as $\alpha$.

For $w\in\F_2\setminus\{e\}$, let $Z_w$ be the set of
$\phi\in[T]$ for which $\beta_\phi(w)$ is weakly mixing.
We claim that $Z_w$ is comeager in $[T]$. Write $w=uvu^{-1}$ with $v$ cyclically reduced. %The
%transformations $\beta_\phi(w)$ and $\beta_\phi(v)$ are conjugate in
%$\Aut(X,\nu)$, so one is weakly mixing if and only if the other is, and 
By passing to a conjugate automorphism we may assume that $w=v$. If $v=a^{k}$ with $k\neq0$, then
$\beta_\phi(v)=T^{k}$ is mixing for every $\phi$. Otherwise $b^{\pm1}$
occurs in $v$, and \cite[Lemmas~3.1 and~3.2]{wrobel2026orbit} together with
the Baire category theorem show that the set $\cE_v^{C_v}\subseteq[T]$ of
\cite[Section~3]{wrobel2026orbit} is comeager; by
\cite[Proposition~2.1]{wrobel2026orbit}, $\beta_\phi(v)$ is weakly mixing
for every $\phi\in\cE_v^{C_v}$. This proves the claim.

Intersecting comeager sets $Z_w$ over
the countable set $\F_2\setminus\{e\}$, we obtain $\phi$ for which the action $\rho=\beta_\phi$ is pointwise weakly mixing.

Finally, every $\rho$-invariant factor is invariant under
$\rho(a)=T$. Since $T$ is prime, such a factor is trivial or the whole
system. Hence $\rho$ is prime. This finishes the proof.
\end{proof}

Similar actions of all nonabelian countable free groups can be obtained by restriction. To that end it suffices to note that for every $n \in \N\cup \{\infty\}$ the group $\F_2=\langle a,b\rangle$ admits free subgroups $H_n \cong\F_n$ such that $a \in H_n$, and $a$ can be viewed as one of the free generators. 

 For $k\in\Z$, let
$ a_k=b^k a b^{-k}$,
and let $\pi$ be the homomorphism from $\F_2$ to $\Z$ given by setting $\pi(a)=0$ and
$\pi(b)=1$. Then $\ker\pi$ is a free group with free basis
$\{a_k\mid k\in\Z\}$. This can be shown using the standard Nielsen--Schreier techniques (see for example \cite[Section~1.3]{Johnsonbook}). Indeed, $\{b^k\mid k\in\Z\}$ is a \emph{Schreier transversal} for
$\ker\pi$ with respect to the generating set $\{a,b\}$, and the associated
Schreier generators are $b^kab^{-k}=a_k$ for the letter $a$ and the
identity for the letter $b$; by the Nielsen--Schreier theorem in its
Reidemeister--Schreier form, $\ker\pi$ is free on the elements $a_k$,
$k\in\Z$. Consequently, every subset of $\{a_k\mid k\in\Z\}$ generates a
subgroup which is free on that subset. Thus we can choose
$H_n=\langle a_0,a_1,\ldots,a_{n-1}\rangle$ for $2\leq n<\infty$ and
$H_\infty=\langle a_k\mid k\geq0\rangle$. 
%Then $H_n\cong\F_n$ for $2\leq n<\infty$ and $H_\infty\cong\F_\infty$.
%Moreover $H_2\leq H_3\leq\cdots\leq H_\infty$, and every $H_n$ contains
%$a_0=a$.

\begin{thm}\label{thm:all-free-ranks}
For every $2\leq n\leq\infty$, there exists an essentially free, prime, pointwise weakly
mixing p.m.p.\ action of $\F_n$. These actions may be chosen as restrictions
of one action of $\F_2$ to the nested free subgroups
$H_2\leq H_3\leq\cdots\leq H_\infty$ defined above.
\end{thm}

\begin{proof}
Identify $\F_n$ with $H_n$ and restrict the action $\rho$ from
Proposition~\ref{prop:F2-prime-twm} to $H_n$. %The restricted action is free
%because $\rho$ is free. It is totally weakly mixing because every infinite
%subgroup of $H_n$ is also an infinite subgroup of $\F_2$. 
If
$\mathcal A$ is an $H_n$-invariant factor, then it is invariant under
$a\in H_n$, hence under the prime transformation $\rho(a)=T$. Therefore
$\mathcal A$ is trivial or the whole system and  the restricted action is prime.
The rest is an easy check.
%It is in particular ergodic, since $\rho(a)=T$ is mixing.
%The proof is complete.
\end{proof}
We next construct non-prime examples for which all factors are coordinate
factors, using the results of \cite{Rudolph_MSJ}.
We shall use the mixing transformation $T$ on a standard nonatomic probability
space $(X_0,\nu_0)$ constructed in
\cite[Theorems~2.7 and~2.8]{Rudolph_MSJ}. We first recall the precise joining
property of this transformation. Let $j \in \N$. An \emph{off-diagonal measure} on $X_0^j$
is the image of $\nu_0$ under a map
$x\mapsto(T^{m_1}x,\ldots,T^{m_j}x)$, where $m_1,\ldots,m_j\in\Z$.
A product of off-diagonal measures is obtained by partitioning the
coordinates into blocks, assigning an off-diagonal measure to each block,
and taking the product over the blocks. A singleton block carries the measure
$\nu_0$.

When $T$ is the automorphism constructed in \cite{Rudolph_MSJ}, for every $n\geq1$ and every choice of nonzero integers $l_1,\ldots,l_n$,
each ergodic $T^{l_1}\times\cdots\times T^{l_n}$-invariant measure with all
one-dimensional marginals equal to $\nu_0$ is such a product. Moreover,
coordinates in the same block have the same exponent $l_i$. This is the
property now called \emph{minimal power joinings}, following
\cite[pp.~533 and~554]{delJuncoRudolph1987graphs}. Rudolph calls it having \emph{minimal
self-joinings}. His definition in \cite[p.~98]{Rudolph_MSJ} allows even
countably many coordinates, so it applies to every finite $n$ used here.
The restriction to equal exponents is also stated at the end of the proof
of \cite[Theorem~2.7]{Rudolph_MSJ}.

\begin{samepage}
\begin{lem}\label{lem: non-prime Z-actions with all factor maps controlled}
For every integer $k\geq1$, there are mixing prime transformations
$T_i\in\Aut(X_i,\nu_i)$ on standard nonatomic probability spaces,
$1\leq i\leq k$, such that every factor algebra of
$T_1\times\cdots\times T_k$ is a coordinate algebra
$\mathcal A_I=\pi_I^{-1}\bigl(\bigotimes_{i\in I}\mathcal B_i\bigr)$
for some $I\subseteq\{1,\ldots,k\}$, where $\mathcal B_i$ is the measure
algebra of $(X_i,\nu_i)$ and $\pi_I$ denotes the corresponding coordinate
projection. Here $\mathcal A_\varnothing$ is the trivial algebra, and all
equalities of factor algebras are understood modulo null sets.
\end{lem}
\end{samepage}

\begin{proof}
Choose distinct positive integers $r_1<\cdots<r_k$, let
$(X_i,\nu_i)=(X_0,\nu_0)$ and $T_i=T^{r_i}$, and write
$R=T_1\times\cdots\times T_k$ and
$(X,\nu)=(X_0^k,\nu_0^{\otimes k})$. Each $T_i$ is mixing, and so is $R$.
Indeed, mixing of a finite product follows first for measurable rectangles
and then for arbitrary measurable sets by approximation.

Let $\mathcal A$ be an $R$-invariant factor algebra and let
$\mathbb{E}_{\mathcal A}$ be the corresponding conditional expectation, acting on $L^2(X,\nu)$.
We consider the \emph{relatively independent self-joining} $\lambda$ over
$\mathcal A$, which is a measure on $X^2=X_0^{2k}$. Both marginals of $\lambda$ are $\nu$, and for
$f,h\in L^2(X,\nu)$ we have
\[
 \int_{X\times X}f(x)\overline{h(y)}d\lambda(x,y)
 =\langle\mathbb{E}_{\mathcal A}f,\mathbb{E}_{\mathcal A}h\rangle
 =\langle\mathbb{E}_{\mathcal A}f,h\rangle.
\]
Since $\mathcal A$ is $R$-invariant, we have
$\mathbb{E}_{\mathcal A}(f\circ R)=(\mathbb{E}_{\mathcal A}f)\circ R$
for every $f\in L^2(X,\nu)$. It follows that
\[
 \begin{aligned}
 \int_{X\times X}f(Rx)\overline{h(Ry)}d\lambda(x,y)
 &=\int_X\mathbb{E}_{\mathcal A}(f\circ R)(x)
       \overline{\mathbb{E}_{\mathcal A}(h\circ R)(x)}d\nu(x)\\
 &=\int_X(\mathbb{E}_{\mathcal A}f)(Rx)
       \overline{(\mathbb{E}_{\mathcal A}h)(Rx)}d\nu(x)\\
 &=\int_X(\mathbb{E}_{\mathcal A}f)(x)
       \overline{(\mathbb{E}_{\mathcal A}h)(x)}d\nu(x)\\
 &=\int_{X\times X}f(x)\overline{h(y)}d\lambda(x,y).
 \end{aligned}
\]
Here the third equality uses that $R$ preserves $\nu$.
Thus, $\lambda$ is $R\times R$-invariant. Let
$\lambda=\int\lambda_\alpha d\eta(\alpha)$ be its ergodic decomposition.

Almost every component $\lambda_\alpha$ has marginal $\nu_0$ on each
of the $2k$ coordinates. Indeed, let $\pi_i$ be the projection onto the
$i$-th coordinate of either copy of $X$. Since
$\pi_i\circ(R\times R)=T^{r_i}\circ\pi_i$, we have
\[
 (T^{r_i})_*(\pi_i)_*\lambda_\alpha
 =(\pi_i)_*(R\times R)_*\lambda_\alpha
 =(\pi_i)_*\lambda_\alpha.
\]
Thus, $(\pi_i)_*\lambda_\alpha$ is $T^{r_i}$-invariant, and
$\nu_0=\int(\pi_i)_*\lambda_\alpha d\eta(\alpha)$.
Since $T^{r_i}$ is ergodic, $\nu_0$ is an extreme point of the convex
set of $T^{r_i}$-invariant probability measures. Hence
$(\pi_i)_*\lambda_\alpha=\nu_0$ for almost every $\alpha$.

We now apply the minimal power joining property to the $2k$ coordinates
of $\lambda_\alpha$ for almost every $\alpha$. The exponents are $r_1,\ldots,r_k,r_1,\ldots,r_k$, and
each one-dimensional marginal is $\nu_0$. Since the $r_i$ are distinct,
every block is either a singleton or the pair $\{i,k+i\}$ for some $i$.
Identifying $X\times X$ with $\prod_{i=1}^k(X_i\times X_i)$ by pairing
the corresponding coordinates, we obtain
$\lambda_\alpha=\bigotimes_{i=1}^k\theta_{i,\alpha}$, where
$\theta_{i,\alpha}$ is a measure on $X_i\times X_i$.
If the two coordinates form singleton blocks, then
$\theta_{i,\alpha}=\nu_0\otimes\nu_0$. Otherwise $\theta_{i,\alpha} =
(T^a,T^b)_*\nu_0=(\id,T^{b-a})_*\nu_0$ for some $a,b\in\Z$.
Thus, each $\theta_{i,\alpha}$ is either product measure or a graph joining
$(\id,T^m)_*\nu_0$ with $m\in\Z$. In particular, both marginals of
$\lambda_\alpha$ are $\nu$.

To use this description, we associate an operator with each joining.
If $\theta$ is a probability measure on $X\times X$ with both marginals
$\nu$, then the Cauchy--Schwarz inequality gives
$\left|\int_{X\times X}f(x)\overline{h(y)}d\theta(x,y)\right|
\leq\|f\|_2\|h\|_2$ for $f,h\in L^2(X,\nu)$. Hence there is a unique contraction $K_\theta$
on $L^2(X,\nu)$ satisfying
$\langle K_\theta f,h\rangle=\int f(x)\overline{h(y)}d\theta(x,y)$.
In particular, $K_\lambda=\mathbb{E}_{\mathcal A}$.

We can also `localise' this construction to measures on $X_0 \times X_0$ with both marginals equal to $\nu_0$.
Let $U_Tf=f\circ T^{-1}$ and
$\mathbb{E}_{\C1}f=(\int f d\nu_0)1$ on $L^2(X_0,\nu_0)$.
The product measure gives the operator $\mathbb{E}_{\C1}$, since its matrix
coefficient is $(\int f d\nu_0)\overline{(\int h d\nu_0)}$.
The graph joining $(\id,T^m)_*\nu_0$ gives $U_T^m$, since
\[
 \int f(z)\overline{h(T^mz)}d\nu_0(z)
 =\langle f,h\circ T^m\rangle
 =\langle U_T^mf,h\rangle.
\]
For elementary tensors $f=f_1\otimes\cdots\otimes f_k$ and
$h=h_1\otimes\cdots\otimes h_k$, the product description of
$\lambda_\alpha$ therefore gives
\[
 \begin{aligned}
 \langle K_{\lambda_\alpha}f,h\rangle
 &=\prod_{i=1}^k\int f_i(x_i)\overline{h_i(y_i)}d\theta_{i,\alpha}(x_i,y_i)\\
 &=\prod_{i=1}^k\langle K_{i,\alpha}f_i,h_i\rangle,
 \end{aligned}
\]
where $K_{i,\alpha}\in\{\mathbb{E}_{\C1}\}\cup\{U_T^m\mid m\in\Z\}$.
Since the linear span of elementary tensors is dense, this proves
$K_{\lambda_\alpha}=K_{1,\alpha}\otimes\cdots\otimes K_{k,\alpha}$.

The ergodic decomposition gives
$\langle\mathbb{E}_{\mathcal A}f,h\rangle
=\int\langle K_{\lambda_\alpha}f,h\rangle d\eta(\alpha)$
for every $f,h\in L^2(X,\nu)$. There are only countably many possible
operators $K_{\lambda_\alpha}$. For each such operator $K$, the set
$\{\alpha\mid K_{\lambda_\alpha}=K\}$ is measurable. Indeed, the matrix
coefficients are measurable, and equality of bounded operators on the
separable space $L^2(X,\nu)$ can be tested on a countable dense subset.
Let $p_K$ be the $\eta$-measure of this set. Then $p_K\geq0$,
$\sum_Kp_K=1$ and $\mathbb{E}_{\mathcal A}=\sum_Kp_KK$.
The series converges in operator norm because
$\sum_Kp_K\|K\|\leq1$.

If $\mathbb{E}_{\mathcal A}f=f$, contractivity of the operators $K$ gives
\[
 \begin{aligned}
 \sum_Kp_K\|Kf-f\|_2^2
 &\leq\sum_Kp_K\bigl(2\|f\|_2^2-2\operatorname{Re}\langle Kf,f\rangle\bigr)\\
 &=2\|f\|_2^2-2\operatorname{Re}\langle\mathbb{E}_{\mathcal A}f,f\rangle
 =0.
 \end{aligned}
\]
Thus, $Kf=f$ whenever $p_K>0$. Conversely, a vector fixed by every such
$K$ is fixed by their convex combination. Consequently,
$L^2(\mathcal A)=\bigcap_{p_K>0}\operatorname{Fix}(K)$.

It remains to identify the fixed space of each $K$. Write
$K=K_1\otimes\cdots\otimes K_k$, and let
$J_K=\{i\mid K_i\neq\mathbb{E}_{\C1}\}$ and
$I_K=\{i\mid K_i=\id\}$. For $i\in J_K$, write $K_i=U_T^{m_i}$.
Here $m_i=0$ exactly when $i\in I_K$, since $T$ is mixing on a nonatomic
space. Let $V_K$ be the Koopman operator of the transformation which
applies $T^{m_i}$ in coordinate $i\in J_K$ and fixes the other coordinates.
Then $K=V_K\mathbb{E}_{\mathcal A_{J_K}}$. Indeed, with the coordinates in
their original order,
\[
 \mathbb{E}_{\mathcal A_{J_K}}(f_1\otimes\cdots\otimes f_k)(x)
 =\left(\prod_{i\notin J_K}\int f_i d\nu_0\right)
   \prod_{i\in J_K}f_i(x_i).
\]
Applying $V_K$ replaces $f_i(x_i)$ by $f_i(T^{-m_i}x_i)$ for $i\in J_K$,
which gives exactly $K(f_1\otimes\cdots\otimes f_k)$. The identity follows
by density.

Suppose $Kf=f$. Since $V_K$ is unitary,
$\|\mathbb{E}_{\mathcal A_{J_K}}f\|_2=\|Kf\|_2=\|f\|_2$.
The operator $\mathbb{E}_{\mathcal A_{J_K}}$ is an orthogonal projection,
so $\mathbb{E}_{\mathcal A_{J_K}}f=f$. Thus, $f$ depends only on the
coordinates in $J_K$ and is fixed by $V_K$.
By Fubini, for almost every value of the coordinates in $I_K$, the
corresponding section of $f$ is invariant under
$\prod_{i\in J_K\setminus I_K}T^{m_i}$.
Each exponent in this product is nonzero, so each transformation is
mixing. Their finite product is therefore mixing and hence ergodic.
Thus, each such section is constant almost everywhere, which shows that
$f\in L^2(\mathcal A_{I_K})$.
If $J_K\setminus I_K$ is empty, this conclusion already follows from
$f\in L^2(\mathcal A_{J_K})$.
Conversely, $K$ fixes every function of the coordinates in $I_K$, since
each of its other tensor factors fixes the constant function $1$.
Therefore $\operatorname{Fix}(K)=L^2(\mathcal A_{I_K})$.

Let $I=\bigcap_{p_K>0}I_K$. For coordinate conditional expectations we
have $\mathbb{E}_{\mathcal A_J}\mathbb{E}_{\mathcal A_L}
=\mathbb{E}_{\mathcal A_{J\cap L}}$.
This follows by testing on elementary tensors, since both sides retain
precisely the coordinates in $J\cap L$ and integrate the others.
%These projections commute, and the range of a product of commuting
%orthogonal projections is the intersection of their ranges.
Only finitely many different subsets $I_K$ can occur, so
$L^2(\mathcal A)=\bigcap_{p_K>0}L^2(\mathcal A_{I_K})=L^2(\mathcal A_I)$.
Equality of these spaces implies equality of the factor algebras modulo
null sets. Indeed, the indicator of a set in either algebra is measurable
with respect to the other algebra up to a null set.
Thus, $\mathcal A=\mathcal A_I$. Conversely, every coordinate algebra is
$R$-invariant, so these are exactly its factor algebras.

Finally, applying the same argument with $k=1$ and $r_1=r$ shows that,
for every $r\geq1$, the only factors of $T^r$ are the trivial factor and
the whole space. Hence every $T^r$ is prime. Each $T_i$ is also mixing,
which proves the lemma.
\end{proof}

\begin{thm}\label{thm:all-free-ranks-nonprime-version}
Let $(X,\nu)$ be a standard nonatomic probability space. For every $2\leq n\leq\infty$, there exists an essentially free, non-prime, pointwise
weakly mixing p.m.p.\ action
$\beta$ of $\F_n$ on $(X,\nu)$ which is also pointwise weakly mixing
relative to every proper $\F_n$-factor of $(X,\nu)$. Moreover, $\beta$
is essentially free on every non-trivial factor.
\end{thm}

\begin{proof}
Fix an integer $k\geq2$. By
Lemma~\ref{lem: non-prime Z-actions with all factor maps controlled},
we can choose mixing prime transformations $T_i$ on $(X_i,\nu_i)$,
$1\leq i\leq k$, whose product has exactly the coordinate factor
algebras. Fix a free generator $a_1$ of $\F_n$. The construction in
the proof of Theorem~\ref{thm:all-free-ranks} gives, for each $i$, a
free, prime, totally weakly mixing action
$\beta_i$ of $\F_n$ on $(X_i,\nu_i)$ with $\beta_i(a_1)=T_i$.
Let
$(X,\nu)=\bigl(\prod_{i=1}^kX_i,\bigotimes_{i=1}^k\nu_i\bigr)$
and $\beta=\beta_1\times\cdots\times\beta_k$.
Then $\beta$ is pointwise weakly mixing, since each
$\beta_i$ is pointwise weakly mixing. %Lemma~\ref{lem:twm-elementwise} therefore
%shows that $\beta$ is totally weakly mixing. 
It is essentially free because the
fixed-point set of $\beta_g$ is contained in the inverse image of the
null fixed-point set of $(\beta_1)_g$ for every $g\neq e$.

Every $\beta$-invariant factor algebra is invariant under
$\beta(a_1)=T_1\times\cdots\times T_k$, and is therefore a coordinate
algebra $\mathcal A_I$ for some $I\subseteq\{1,\ldots,k\}$. Conversely,
each $\mathcal A_I$ is invariant under the product action $\beta$.
Thus, the factors of $\beta$ are exactly its coordinate factors.
Since $k\geq2$ and the coordinate spaces are nonatomic, a
one-coordinate factor is non-trivial and proper, so $\beta$ is not
prime.

Fix $g\in\F_n\setminus\{e\}$ and a proper subset
$I\subsetneq\{1,\ldots,k\}$. Write
$X_I=\prod_{i\in I}X_i$ and $\nu_I=\bigotimes_{i\in I}\nu_i$, with
the empty product interpreted as the one-point probability space.
The relative square of $X$ over its coordinate factor $X_I$ identifies
with $(X_I\times X_{I^c}\times X_{I^c},
\nu_I\otimes\nu_{I^c}\otimes\nu_{I^c})$.
Under this identification, the induced transformation is
$\beta_{I,g}\times\beta_{I^c,g}\times\beta_{I^c,g}$, where
$\beta_{J,g}=\prod_{i\in J}(\beta_i)_g$. This transformation is
ergodic, being a finite product of weakly mixing transformations
and possibly a one-point factor. Hence $\beta_g$ is weakly mixing
relative to $\mathcal A_I$, by
Remark~\ref{remark: relative weak mixing for diagonal action wrt a coordinate factor}.

Finally, if $I\neq\varnothing$, choose $i\in I$. For every $g\neq e$,
the fixed-point set of $\beta_{I,g}$ is contained in the inverse image
of $\operatorname{Fix}((\beta_i)_g)$ and consequently has measure zero.
Thus, every non-trivial factor is essentially free.
\end{proof}

\bibliography{30thsept}
\bibliographystyle{plain}
\end{document}